\documentclass[12pt]{amsart}
\usepackage{a4wide}
\usepackage{amsmath}
\usepackage{amsfonts}
\usepackage{amssymb}
\usepackage{hyperref}
\usepackage{mathrsfs}
\usepackage{pb-diagram}
\usepackage{color}
\usepackage{dynkin-diagrams}
\usepackage{epstopdf}
\usepackage{enumerate}
\usepackage{amsthm}
\usepackage{bbm}
\usepackage{tikz}
\usepackage{tikz-cd}
\usepackage{cleveref}
\usepackage{stmaryrd}
\usepackage[all]{xypic}
\usepackage{epstopdf}
\newtheorem{lemma}{Lemma}[section]
\newtheorem{remark}[lemma]{Remark}
\newtheorem{theorem}[lemma]{Theorem}
\newtheorem{corollary}[lemma]{Corollary}

\newtheorem{definition}[lemma]{Definition}
\newtheorem{proposition}[lemma]{Proposition}
\newtheorem{hypothesis}[lemma]{Hypothesis}

\usepackage[style=alphabetic,citestyle=alphabetic,maxnames=99,maxcitenames=99, maxalpha names=99, sorting=nyt]{biblatex}
\usepackage{dsfont}
\usepackage{comment}
\hypersetup{colorlinks,%
    citecolor=brown,%
    filecolor=black,%
   linkcolor=blue,%
   urlcolor=black}

\newcommand{\R}{\mathbb{R}}

\newcommand{\Z}{\mathbb{Z}}

\renewcommand{\subset}{\subseteq}
\renewcommand{\supset}{\supseteq}

\newcommand{\sgn}{\mathrm{sgn}}

\newcommand{\mca}[1]{\mathcal{#1}}

\newcommand{\End}{\mathrm{End}}
\newcommand{\Hom}{\mathrm{Hom}}

\newcommand{\WF}{\mathrm{WF}}

\newcommand{\s}{\mathbf{s}}

\begin{document}

	\title{The stable wave front set of theta representations}
	
	%\begin{abstract}
	%\end{abstract}
	
	\begin{abstract}
            We compute the stable wave front set of theta representations for certain tame Brylinski-Deligne covers of a connected reductive $p$-adic group. The computation involves two main inputs. First we use a theorem of Okada, adapted to covering groups, to reduce the computation of the wave front set to computing the Kawanaka wave front set of certain representations of finite groups of Lie type. Second, to compute the Kawanaka wave front sets we use Lusztig's formula. This requires a careful analysis of the action of the pro-$p$ Iwahori-Hecke algebra on the theta representation, using the structural results about Hecke algebras developed by Gao-Gurevich-Karasiewicz and Wang.
	\end{abstract}

	\author{Edmund Karasiewicz}
	\address{Department of Mathematics, National University of Singapore, Singapore.}
	\email{karasiee@nus.edu.sg}
	
	\author{Emile Okada}
	\address{Department of Mathematics, National University of Singapore, Singapore.}
	\email{okada@nus.edu.sg}

 	\author{Runze Wang}
	\address{School of Mathematical Sciences, Zhejiang University, 866 Yuhangtang Road, Hangzhou, China 310058}
	\email{wang\underline{\ }runze@zju.edu.cn}

	\date{\today}
	
	\maketitle

%%%%%%%%%%%%%%%%%%%%%%%%%%%%%%%%%%%%%%%%%%%%%%%
%%%%%%%%%%%%%%%%%%%%%%%%%%%%%%%%%%%%%%%%%%%%%%%
%INTRO
%%%%%%%%%%%%%%%%%%%%%%%%%%%%%%%%%%%%%%%%%%%%%%%
%%%%%%%%%%%%%%%%%%%%%%%%%%%%%%%%%%%%%%%%%%%%%%%

\section{Introduction}

Our main result is a computation of the stable wave front set of theta representations of certain tame Brylinski-Deligne covers of a connected reductive $p$-adic group. In particular, this proves a conjecture of Gao-Tsai \cite[Conjecture 2.5, equation (2.3)]{GT22}. We begin with some background and motivation. Then we state our main theorem and outline the proof.

The higher theta representations were introduced by Kubota \cite{K69}.  Motivated by the connection between quadratic reciprocity and classical theta functions (known to Hecke and possibly earlier), Kubota was interested in developing a connection between automorphic forms and $n$-th power reciprocity. Kubota's insight was to focus on the construction of the classical theta functions as residues of Eisenstein series on a $2$-fold cover of $\mathrm{GL}_{2}$. With this reference point Kubota constructed an $n$-fold cover of $\mathrm{GL}_{2}$ and defined the higher theta functions to be residues of Eisenstein series on this $n$-fold cover.

One of the fundamental problems in the study of Kubota's higher theta functions and their generalizations to other covering groups is the determination of their Fourier coefficients. This problem remains open even in the case of $\mathrm{GL}_{2}^{(n)}$ an $n$-fold cover of $\mathrm{GL}_{2}$, where $n\geq4$. A few works in this area are \cite{P77I,P77II,S83,S93, KP84,BH89,EP92,CFH12,FG17,FG18, BH16}. For a nice survey see Hoffstein \cite{H93}.

While the determination of the Fourier coefficients of higher theta functions appears to be a deep and difficult problem, knowledge of these coefficients is important in several applications. We mention two. First, Patterson \cite{P78}, Heath-Brown-Patterson \cite{HBP79}, and more recently Dunn-Radziwill \cite{DR24} have used cubic theta functions on $\mathrm{GL}_{2}$ to study the distribution of cubic Gauss sums. Second, Friedberg-Ginzburg \cite{FG20} have generalized the classical theta correspondence between symplectic-orthogonal dual pairs to higher covers of symplectic and orthogonal groups. Their construction requires knowing the wave front set of higher theta functions of symplectic groups. Our main theorem proves the analogous local result, but the global result needed for Friedberg-Ginzburg's construction remains open. We note however that the local wave front set gives an upper bound on the global wave front set. We say a bit more about this at the end of this introduction.

In addition to the global considerations mentioned above the local theta representations (the local factors of the automorphic representation generated by the theta functions) are interesting in their own right. Under a Shimura correspondence (e.g. \cite{S04}), the unramified theta representation corresponds to the trivial representation of a suitable linear group. Thus these theta representations are fundamental and should be examined completely. 

As was already noted when mentioning the work of Friedberg-Ginzburg, one important invariant to consider is the wave front set. In this paper we compute the stable wave front set of local theta representations.

Now we introduce the notation needed to state our main theorem. Precise definitions can be found in the main body of the paper. Let $F$ be a $p$-adic field containing $\mu_{n}$ the full group of $n$-th roots of unity, where $\mathrm{GCD}(p,n)=1$. We let $G$ be the $F$-points of a connected reductive group and $\overline{G}$ be an $n$-fold cover of $G$ arising from a Brylinski-Deligne extension. A theta representation $\Theta$ for $\overline{G}$ will be the Langlands quotient of a certain exceptional principal series. For example, when $n=1$, $\Theta$ is the trivial representation of $\overline{G}=G$, and for the metaplectic double cover of the symplectic group $\Theta$ is the even Weil representation. One note on terminology. Our use of `theta representation' may differ from usage elsewhere in the literature. For example, in Ginzburg-Rallis-Sourdy \cite{GRS97} the local factors of their theta representations are not among the representations that we study in this paper. 

Let $T\subset G$ be a maximal torus and $W$ the Weyl group of $(G,T)$. For $W^{\prime}\subset W$ a reflection subgroup let $j_{W^{\prime}}^{W}(\mathrm{sgn})\in\mathrm{Irr}(W)$ denote $j$-induction of the sign representation of $W^{\prime}$. Let $\overline{T}$ be the full preimage in $\overline{G}$. We write $\mathrm{WF}^{s}(\Theta)$ for the stable wave front set of $\Theta$. For $\sigma\in \mathrm{Irr}(W)$ let $\mathcal{O}_{\mathrm{Spr}}(\sigma)$ be the nilpotent orbit part of the Springer correspondence applied to $\sigma$. Now we state our main theorem. 

\begin{theorem}[Theorem \ref{WFThetaThm}]\label{WFThetaThmIntro}Suppose that hypotheses \ref{hyp:exp}, \ref{hyp:lift} and \ref{hyp:bilin} hold.
This holds for example if the residue characteristic is larger than some constant depending on the absolute root datum of $\mathbf G$ and the ramification index of $F/\mathbb Q_p$. 

Fix $\pi^{\dagger}$ a distinguished genuine irreducible $\overline{T}$-representation and let $\nu:T\rightarrow \mathbb{C}^{\times}$ be an exceptional character. Let $\Theta$ be the theta representation of $\overline{G}$ attached to $\pi^{\dagger}$ and $\nu$. Then
\begin{equation}\label{WFThetaIntro}
\mathrm{WF}^{s}(\Theta)=\mathcal{O}_{\mathrm{Spr}}(j_{W_{\nu}}^{W}(\mathrm{sgn})).
\end{equation}
\end{theorem}

We note that the right hand side of \eqref{WFThetaIntro} is explicitly computed in Gao-Liu-Tsai \cite{glt}.

Special cases of Theorem \ref{WFThetaThmIntro} were previously known. In type $A$ the theorem is due to Savin (unpublished) and Cai \cite{Cai19}. In type $B$, the case of $n=2$ is due to Bump-Friedberg-Ginzburg \cite{BFG03}. In type $C$, the case of $n=4$ is due to Leslie \cite{L19}; Friedberg-Ginzburg \cite{FG20} conjectured the result for odd-fold covers of $\mathrm{Sp}_{2r}$ and proved partial results. For type $G_{2}$, the case of $n=3$ was proved by Savin \cite{S93} as a byproduct of his construction of the global minimal representation. For $2$-fold covers Loke-Savin \cite{LS10} construct and study global theta functions. In particular, their work implies the case of the $2$-fold cover of $F_{4}$. As previously mentioned, Gao-Tsai \cite[Conjecture 2.5]{GT22} made the general conjecture of equation \eqref{WFThetaIntro} based on these previously known examples and their computation of the archimedean analog.

The proof of Theorem \ref{WFThetaThmIntro} breaks into three steps. 
\begin{enumerate}
	\item We apply Okada's Theorem to reduce the computation of $WF^{s}(\Theta)$ to computing Kawanaka wave front sets of representations of finite groups of Lie type. \label{Step1}
 %These representations arise from the fixed vectors of the pro-$p$ radical of parahoric subgroups.
	\item We apply Lusztig's formula for the Kawanaka wave front \cite{L92} set to reduce to a combinatorial computation.\label{Step2}
	\item We do a case by case analysis to complete the combinatorial computation.\label{Step3}
\end{enumerate}

 Next we outline of the contents of each section. Along the way we add more details to the three steps in our proof.

Section \ref{Prelim} consists of preliminary material. Aside from setting our notation we also extend Okada's Theorem \cite{okada} to the case of tame Brylinski--Deligne covers, Theorem \ref{StableWFS}.
It expresses the stable wave front set as the maximum over orbits derived from the Kawanaka wave front sets of $\Theta^{P^+}$ where $\overline P$ is a parahoric subgroup of $\overline G$ and $P^+$ denotes the the pro-$p$ radical of $\overline P$.

Section \ref{Theta} contains steps \eqref{Step1} and \eqref{Step2} in our proof of Theorem \ref{WFThetaThmIntro}. We begin in Subsection \ref{ThetaPropSS} with the definition of the theta representation $\Theta$. Propositions \ref{ThetaBasics} and \ref{ThetaPmod} collect the relevant properties of $\Theta$. 
These properties capture the structure of the $\overline P$-representation $\Theta^{P^+}$.
By Moy--Prasad theory the $\overline P/P^+$-representation $\Theta^{P^+}$ is a sum of consistuents of principal series. 
%The parahorics act through a particular constituent of some principal series of their reductive quotients. 
Knowing this is crucial to applying Lusztig's formula for the Kawanaka wave front set later in this section.

In Subsection \ref{KWFSGen}, we use Lusztig's formula to compute the Kawanaka wave front set of certain constituents of principal series of finite groups of Lie type.
The main result is Corollary \ref{KWFS}.

In Subsection \ref{KWFSComp} we use Corollary \ref{KWFS} to compute the Kawanaka wave front set of the $\overline P/P^+$-representation $\Theta^{P^+}$.
The main result of this subsection is Corollary \ref{ParaKWFS}.

%we translate between the language of parahorics and finite groups of Lie type so that we can apply Lusztig's formula. The main result of this subsection is Corollary \ref{ParaKWFS}.

We put everything together in Subsection \ref{CombRed} to complete Step \eqref{Step1} and \eqref{Step2}. Implementing these two steps yields the main result of this subsection Proposition \ref{WFThetaRedtoComb}. This proposition gives a formula for the stable wave front set of $\Theta$ as a maximum over nilpotent orbits derived from root theoretic data. 
%We add a few details. 
%Step \eqref{Step1} We apply Okada's Theorem (Theorem \ref{StableWFS}). This tells us that the stable wave front set of $\Theta$ is the maximum over orbits derived from Kawanaka wave front sets of particular constituents of principal series of finite groups of Lie type. The particular constituent of the principal series is determined by the properties of $\Theta$ from Subsection \ref{ThetaPropSS}. Step \eqref{Step2} These Kawanaka wave front sets are computed using Lusztig's formula, which we made explicit in Subsection \ref{KWFSComp}. 

Section \ref{WFSComp} contains the case by case analysis of Step \eqref{Step3} in the proof of Theorem \ref{WFThetaThmIntro}.

We conclude this introduction with a remark on the global significance of Theorem \ref{WFThetaThmIntro}. By general arguments and the work of Moeglin-Waldspurger \cite{MW87} the wave front set of a global automorphic representation is bounded above by the wave front set of any of its local factors. Thus Theorem \ref{WFThetaThmIntro} gives an upper bound for the stable wave front set of global theta representations. This generalizes a result of Friedberg-Ginzburg who proved the upperbound in type $C$ \cite[Theorem 1, (1)]{FG20}.

%%%%%%%%%%%%%%%%%%%%%%%%%%%%%%%%%%%%%%%%%%%%%%%
%%%%%%%%%%%%%%%%%%%%%%%%%%%%%%%%%%%%%%%%%%%%%%%
%Preliminaries
%%%%%%%%%%%%%%%%%%%%%%%%%%%%%%%%%%%%%%%%%%%%%%%
%%%%%%%%%%%%%%%%%%%%%%%%%%%%%%%%%%%%%%%%%%%%%%%

\subsection{Acknowledgements}
The authors would like to express their gratitude to Fan Gao for suggesting the problem of computing the wave front set of theta representations to the authors during the ``Arthur packets" conference at The Institute for Advanced Study in Mathematics in Hangzhou, China and sharing many helpful insights in our discussions.
The authors would also like to thank Gordan Savin for valuable comments on this pre-print.

\section{Preliminaries}\label{Prelim}

Fix a prime number $p$. Let $F$ be a $p$-adic field with ring of integers $\mathfrak{o}=\mathfrak{o}_{F}$, maximal ideal $\mathfrak{p}=\mathfrak{p}_{F}$, and residue field $\mathfrak{f}=\mathfrak{o}/\mathfrak{p}$. Let $q$ be the size of the residue field and fix $\varpi\in\mathfrak{p}$ a uniformizer. Let $|-|:F\rightarrow \mathbb{R}_{>0}$ be the absolute value of $F$ normalized so that $|\varpi|=q^{-1}$.

If $\bf{H}$ is a linear algebraic group defined over a commutative ring $R$, then we may write $H_{R}$ for $\mathbf{H}(R)$; and if $R=F$ we simply write $H$. If $H$ is a locally compact topological group we write $\delta_{H}$ for the modular character of $H$.

Let $\mathscr{G}$ be a smooth connected reductive linear algebraic group defined over $\mathfrak{o}$ with generic fiber $\bf{G}$, a split connected reductive group defined over $F$. We write $G_{\mathfrak{o}}$ for $\mathscr{G}(\mathfrak{o})$ and $G_{\mathfrak{f}}$ for the $\mathfrak{f}$-points of the special fiber of $\mathscr{G}$.

Let $\bf{T}$ be a maximal torus of $\bf{G}$. The pair $(\bf{G},\bf{T})$ determines a root datum \begin{equation*}
(X,\Phi,Y,\Phi^{\vee}),
\end{equation*}
where $X$ is the character lattice of $\bf{T}$, $Y$ is the cocharacter lattice of $\bf{T}$, $\Phi$ is the set of roots, and $\Phi^{\vee}$ is the set of coroots.

We write
\begin{equation*}
\langle-,-\rangle:X\times Y\rightarrow \mathbb{Z}
\end{equation*}
for the perfect pairing between $X$ and $Y$. Let $Y^{sc}:=\mathrm{span}_{\mathbb{Z}}(\Phi^{\vee})$.

Fix a Borel subgroup $\bf{B}$ containing $\bf{T}$. This determines a set of positive roots $\Phi^{+}\subset \Phi$ and a set of simple roots $\Delta\subset \Phi^{+}$. Let $\Phi^{-}:=-\Phi^{+}$ be the set of negative roots relative to $\bf{B}$. Let $\rho:=\frac{1}{2}\sum_{\alpha\in\Phi^{+}}\alpha$.

We fix a Chevalley-Steinberg system of pinnings $\{e_{\alpha}:\mathbf{G}_{a}\rightarrow \mathbf{G}|\alpha\in\Phi\}$ with respect to $(\mathbf{G},\mathbf{T})$.

Let $N(\bf{T})\subset \bf{G}$ be the normalizer of $\bf{T}$ in $\bf{G}$, and $W=W(G,T):=N(T)/T$ be the Weyl group of $\bf{T}$.

We write $W_{\mathrm{af}}:=W\ltimes Y^{sc}$ for the affine Weyl group, and $W_{\mathrm{ea}}:=W\ltimes Y$ for the extended affine Weyl group. The set of affine roots $\Phi_{\mathrm{af}}$ consists of the affine functions on the space $Y\otimes \mathbb{R}$ of the form $\alpha+k$, where $\alpha\in\Phi$ and $k\in\mathbb{Z}$. For each $a=\alpha+k\in \Phi_{\mathrm{af}}$, let $H_{a}$ be the vanishing  hyperplane of $a$ and let $s_{a}$ denote the reflection of $Y\otimes \mathbb{R}$ fixing the hyperplane $H_{a}$ defined by $s_{a}(y)=y-(\langle\alpha,y\rangle+k)\alpha^{\vee}$. If we choose a connected component $\mathscr{A}_{0}$ of $Y\otimes \mathbb{R}-\cup_{a\in\Phi_{\mathrm{af}}}H_{a}$, then $W_{\mathrm{af}}$ is a Coxeter group with distinguished generators given by the reflections $s_{a}$, $a\in\Phi_{\mathrm{af}}$, such that $H_{a}$ intersects the closure of the fixed connected component $\mathscr{A}_{0}$.

\subsection{Covering Groups}

We work with a family of groups that arise from Brylinksi-Deligne extensions \cite{BD01}. This will be the same family that appears in \cite{GGK}.

From \cite{W14} (see also \cite[Section 2.6]{GG18}), a Brylinski-Deligne $\mathbf{K}_{2}$-extension of $\mathbf{G}$ 
\begin{equation}\label{CExt}
\mathbf{K}_{2}\hookrightarrow \overline{\mathbf{G}}\stackrel{\wp}{\twoheadrightarrow} \mathbf{G}
\end{equation}
can be incarnated by the pair $(D,\eta)$, where
\begin{itemize}
\item $D:Y\times Y\rightarrow \mathbb{Z}$ is a bilinear form (not necessarily symmetric) such that $Q(y):=D(y,y)$ is a $W$-invariant quadratic form;
\item $\eta:Y^{sc}\rightarrow F^{\times}$ is a group homomorphism.
\end{itemize}
In this paper, as in \cite{GGK}, we assume that $\eta$ is trivial.

For any algebraic subgroup $\mathbf{H}\subset \mathbf{G}$ we write $\overline{\mathbf{H}}$ for the pullback of \eqref{CExt} with respect to the inclusion $\mathbf{H}\subset \bf{G}$.

Any unipotent subgroup of $\mathbf{G}$ has a canonical splitting of the exact sequence \eqref{CExt}. For each $\alpha\in\Phi$ let $\overline{e}_{\alpha}:\mathbf{G}_{a}\rightarrow \overline{\mathbf{G}}$ be the canonical splitting of the root subgroup associated to $\alpha$. With this, for any $\alpha\in \Phi$ and $a\in\mathbf{G}_{m}$ we write
\begin{itemize}
\item $\overline{w}_{\alpha}(a):=\overline{e}_{\alpha}(a)\overline{e}_{-\alpha}(-a^{-1})\overline{e}_{\alpha}(a)\in \overline{N}(\mathbf{T})$;
\item $\overline{h}_{\alpha}(a):=\overline{w}_{\alpha}(a)\overline{w}_{\alpha}(-1)\in \overline{\mathbf{T}}$
\end{itemize}

The quotient map $\overline{\mathbf{T}}\stackrel{\wp}{\twoheadrightarrow} \mathbf{T}$ admits a section $\mathbf{s}$ such that for any $a_{1},a_{2}\in \mathbf{G}_{m}$ and $y_{1},y_{2}\in Y$ 
\begin{equation}
\mathbf{s}(y_{1}(a_{1}))\mathbf{s}(y_{1}(a_{1}))=\{a_{1},a_{2}\}^{D(y_{1},y_{2})}\mathbf{s}(y_{1}(a_{1})y_{2}(a_{2})),
\end{equation}
where $\{a_{1},a_{2}\}\in K_{2}$. (See \cite[Section 0.N.5]{BD01}.) Since  $\eta$ is trivial \cite[Section 3.3]{GG18} implies for all $\alpha\in \Delta$
\begin{equation}\label{CorootSec}
\overline{h}_{\alpha}(a)=\mathbf{s}(h_{\alpha}(a)).
\end{equation}

Since $F$ is a $p$-adic field with finite residue field of order $q$, it follows that $F$ contains $\mu_{q-1}$ the full group of $(q-1)$-st roots of unity. Let $n\in \mathbb{Z}_{\geq 0}$ be such that $n$ divides $q-1$. Note that $\mathrm{GCD}(n,p)=1$.

In the central extension \eqref{CExt} we can take $F$-points and push out by Hilbert symbols to construct the following commutative diagram. 

\begin{equation}\label{CExtsAll}
    \begin{tikzcd}
        \mathbf{K}_{2}(F) \ar[r, hook] \ar[d, two heads] & \overline{\mathbf{G}}(F)\ar[r, two heads] \ar[d, two heads] & G \ar[d, equal] \\
        \mu_{q-1}  \ar[r, hook] \ar[d,two heads,"(-)^{\frac{q-1}{n}}"] & \overline{G}^{u} \ar[r, two heads,"\wp^{u}"] \ar[d,two heads, "\overline{\wp}"] & G\ar[d,equal]\\
        \mu_{n}\ar[r,hook]& \overline{G}\ar[r,two heads, "\wp"]& G.
    \end{tikzcd}
\end{equation}

Specifically, the second row is the push out of the first row along the unramified Hilbert symbol $(-,-)_{\mathrm{un}}:\mathbf{K}_{2}(F)\rightarrow \mu_{q-1}$; the third row is the push out by the $n$-th order Hilbert symbol $(-,-)_{n}:\mathbf{K}_{2}(F)\rightarrow \mu_{n}$. The objective of this work is to study certain representations on the group $\overline{G}$. But at times it will be beneficial to pull back to $\overline{G}^{u}$, which exhibits better algebraic properties (see Subsection \ref{ParaSubs}).

For any $H\subset G$, we write $\overline{H}:=\wp^{-1}(H)$ and $\overline{H}^{u}:=(\wp^{u})^{-1}(H)$. 

Let $B_{Q}:Y\otimes Y\rightarrow \mathbb{Z}$ be the $W$-invariant bilinear form attached to $Q$, defined by $B_{Q}(y,z):=Q(y+z)-Q(y)-Q(z)$. The commutator of $\overline{T}$ descends to $T$ and defines a bimultiplicative map
\begin{equation*}
[-,-]:T\times T\rightarrow \mu_{n}.
\end{equation*}

 Fix an embedding $\varepsilon:\mu_{n}\hookrightarrow\mathbb{C}^{\times}$. If $\pi$ is a representation of $\overline{G}$ such that $\mu_{n}\subset \overline{G}$ acts through the character $\varepsilon$ we say that $\pi$ is $\varepsilon$-genuine. Sometimes we abuse notation and use $\varepsilon$ to identify $\mu_{n}$ with the $n$-th roots of unity in $\mathbb{C}^{\times}$.

 By \cite[Theorem 6.6]{GG18} and because $\eta$ is trivial, $\overline{T}$ admits an $\varepsilon$-genuine representation $\pi^{\dagger}$ determined by a distinguished $\varepsilon$-genuine character $\chi^{\dagger}:Z(\overline{T})\rightarrow \mathbb{C}^{\times}$. Note, $\mathrm{dim}\pi^{\dagger}$ is equal to the index of any maximal abelian subgroup in $\overline{T}$.

 We fix $\pi^{\dagger}$ (equivalently $\chi^{\dagger}$) for the remainder of this work.

\subsection{Modified Root Datum}

For $\alpha\in \Phi$ define 

\begin{align*}
    n_{\alpha}:=&\frac{n}{\mathrm{GCD}(n,Q(\alpha^{\vee}))};\\ 
    \alpha_{Q,n}:=&\frac{1}{n_{\alpha}}\alpha;\\
    \alpha_{Q,n}^{\vee}:=&n_{\alpha}\alpha^{\vee};\\
    Y_{Q,n}:=&\{y\in Y|B_{Q}(y,y^{\prime})\in n\mathbb{Z}\text{ for all }y^{\prime}\in Y\};\\
    \Phi_{Q,n}^{\vee}:=&\{\alpha_{Q,n}^{\vee}|\alpha\in \Phi\};\\
    X_{Q,n}:=&\mathrm{Hom}(Y_{Q,n},\mathbb{Z});\\
    \Phi_{Q,n}:=&\{\alpha_{Q,n}|\alpha\in\Phi\}.
\end{align*}
The modified root datum attached to the covering group $(\overline{G},\overline{T})$ is
\begin{equation*}
(X_{Q,n},\Phi_{Q,n},Y_{Q,n},\Phi_{Q,n}^{\vee}).
\end{equation*}

Let $\mathscr{X}:=Y/Y_{Q,n}$. Since $W$ preserves $Y_{Q,n}$, the Weyl group $W$ acts on $\mathscr{X}$.Thus we have an action of $W\ltimes \mathscr{X}$ on $\mathscr{X}$, where $\mathscr{X}$ acts on itself by translation. Pulling back by the natural surjection $W_{\mathrm{ea}}\simeq W\ltimes Y\twoheadrightarrow W\ltimes \mathscr{X}$ we get an action of the extended affine Weyl group $W_{\mathrm{ea}}$ on $\mathscr{X}$. 

\subsection{Tame Covers}

Since $n$ divides $q-1$, the covering group $\overline{G}$ is called a tame cover of $G$ and in this subsection we collect some properties special to tame covers.

\begin{comment}
Since $\overline{G}$ is tame and $\eta=1$, \cite[Theorem 4.2]{GG18} implies that there exists a splitting of the hyperspecial maximal compact subgroup
\begin{equation*}
    \mathscr{S}:G_{\mathfrak{o}}\rightarrow \overline{G}.
\end{equation*}
\end{comment}

Since $p$ does not divide $n$ the bimultiplicative map $[-,-]:T\times T\rightarrow \mu_{n}$ is trivial on $T_{\mathfrak{o}}^{+}$, the maximal pro-$p$ subgroup of $T_{\mathfrak{o}}$, and thus factors through $T_{\mathfrak{o}}/T_{\mathfrak{o}}^{+}\simeq T_{\mathfrak{f}}$, where the isomorphism is given by mod $p$ reduction. Thus we can define a map
\begin{equation*}
    \varphi:T\rightarrow \mathrm{Hom}(T_{\mathfrak{f}},\mu_{n})
\end{equation*}
by $\varphi(t)(s):=[t,s^{\prime}]$, where $s^{\prime}\in T_{\mathfrak{o}}$ is any lift of $s\in T_{\mathfrak{f}}$.

By restriction we get a homomorphism $Y\rightarrow \mathrm{Hom}(T_{\mathfrak{f}},\mu_{n})$, defined by $y\mapsto \varphi(y(\varpi))$. Note that this is independent of the choice of uniformizer. The kernel is $Y_{Q,n}$. Thus we get an injective map, which we also call $\varphi$
\begin{equation*}
    \varphi:\mathscr{X}\hookrightarrow \mathrm{Hom}(T_{\mathfrak{f}},\mu_{n}).
\end{equation*}

For tame covers, the group $A:=\overline{T}_{\mathfrak{o}}Z(\overline{T})\subset\overline{T}$ is a maximal abelian subgroup of $\overline{T}$. From this it follows that $\mathrm{dim}\pi^{\dagger}=|\mathscr{X}|$.

Let $\mathscr{Z}=\mathscr{Z}_{\pi^{\dagger}}\subseteq \mathrm{Hom}(\overline{T}_{\mathfrak{o}},\mathbb{C}^{\times})$ be the set of characters of $\overline{T}_{\mathfrak{o}}$ contained in $\pi^{\dagger}$. Note that $\mathscr{X}$ acts on $\mathscr{Z}$ via $y\mapsto (\chi\mapsto \varphi(y)\cdot\chi).$ Moreover, since $\pi^{\dagger}$ is Weyl group invariant we see that $\mathscr{Z}$ admits an action of $W\ltimes \mathscr{X}$.

The next lemma says $\mathscr{Z}$ is a torsor under $\mathscr{X}$. The proof is straightforward and omitted. Let $\varepsilon_{*}:\mathrm{Hom}(T_{\mathfrak{o}},\mu_{n})\rightarrow \mathrm{Hom}(T_{\mathfrak{o}},\varepsilon(\mu_{n}))$ be defined by $f\mapsto \varepsilon\circ f$.

\begin{lemma}\label{ZTorUnderX}
    Fix any $\chi_{0}\in \mathscr{Z}$. 
    The map $\mathscr{Z}\rightarrow \mathrm{Im}\varepsilon_{*}\circ\varphi\subset \mathrm{Hom}(T_{\mathfrak{o}},\varepsilon(\mu_{n}))$ defined by $\chi\mapsto \chi\cdot (\chi_{0})^{-1}$ is an $\mathscr{X}$-equivariant bijection. (On the domain and codomain the action of $\mathscr{X}\simeq\mathrm{Im}\varphi$ is by multiplication of characters.) %\textbf{One can define a twisted action of $W$ on $\mathrm{Im}\varphi$ to make this $W\ltimes \mathscr{X}$-equivariant, but maybe we do not need this.}
\end{lemma}

%Let $\mathrm{Hom}_{\varepsilon}(\overline{T}_{\mathfrak{o}},\mu_{n})$ be the set of $\varepsilon$-genuine characters of $\overline{T}_{\mathfrak{o}}$ valued in $\varepsilon(\mu_{n})\subset \mathbb{C}^{\times}$. \textbf{Better place for this?} 

We end this subsection with a lemma showing that the elements of $\mathscr{Z}$ are quite special.

\begin{lemma}\label{SimpTorChar}
Let $\chi\in \mathscr{Z}$. There is a coweight $y_{\chi}\in Y^{sc}\otimes \mathbb{Q}$ such that for all $\alpha\in\Phi$
\begin{equation*}
    \chi(\overline{h}_{\alpha}(u))= (\varpi,u)_{n}^{B_{Q}(y_{\chi},\alpha^{\vee})}.
\end{equation*}
\end{lemma}

\begin{proof}
    For the proof we will assume that $\Phi$ is irreducible. In general one applies the following argument on each irreducible component.

    We begin with some preliminaries. Because $\Phi$ is irreducible there is a unique nondegenerate $W$-invariant bilinear form $(-|-)$ on $Y^{sc}\otimes \mathbb{Q}$ such that for any short coroot $\alpha^{\vee}$ we have $(\alpha^{\vee}|\alpha^{\vee})=1$. By the uniqueness, there exists $m\in \mathbb{Z}$ such that $B_{Q}(-,-)=m(-|-)$. Since $B_{Q}(\alpha^{\vee},\alpha^{\vee})=2Q(\alpha^{\vee})$, we see that $m$ is even, and for all $\alpha\in \Phi$ we have 
    \begin{equation*}
        Q(\alpha^{\vee})=\frac{m}{2}(\alpha^{\vee}|\alpha^{\vee}).
    \end{equation*}
    
    Define $X^{sc}:=\mathrm{Hom}(Y^{sc},\mathbb{Z})$. 
    By \cite[Ch VI, Section 1, no. 1, Lemma 2]{Bour}, the map 
    $$\iota:X^{sc}\otimes \mathbb{Q}\rightarrow Y^{sc}\otimes \mathbb{Q},\quad \alpha\mapsto \frac{2\alpha^{\vee}}{(\alpha^{\vee}|\alpha^{\vee})}$$
    is an isomorphism of $\mathbb{Q}$ vector spaces. 

    For each $\alpha\in \Delta$ we consider $\omega_{\alpha}\in X^{sc}$, the fundamental weight attached to $\alpha$. There exists $y_{\alpha}\in Y^{sc}\otimes \mathbb{Q}$ such that $\iota(\omega_{\alpha})=y_{\alpha}$.
    In particular, for all $y\in Y$, we have 
    $$(y_{\alpha}|y)=\langle\omega_{\alpha},y\rangle\in \mathbb{Z}.$$

    Now we turn to the proof of the lemma.
    
    By \cite[Section 6.4]{GG18}, for all $\alpha\in \Phi$ there exists $k_{\alpha}\in \mathbb{Z}$ such that 
    \begin{equation*}
        \chi(\overline{h}_{\alpha}(u))=\varepsilon((\varpi,u)_{n}^{Q(\alpha^{\vee})k_{\alpha}}).
    \end{equation*}
    
    Consider $x:=\sum_{\alpha\in\Delta}k_{\alpha}\frac{(\alpha^{\vee}|\alpha^{\vee})}{2}\omega_{\alpha}\in X^{sc}$ and let $y_{\chi}:=\iota(x)$. A direct calculation shows that $y_{\chi}$ is a coweight.

    Finally the lemma follows because
    \begin{align*}
        B_{Q}(y_{\chi},-)=&\sum_{\alpha\in \Delta}k_{\alpha}\frac{m}{2}(\alpha^{\vee}|\alpha^{\vee})(y_{\alpha}|-)\\
        =&\sum_{\alpha\in \Delta}k_{\alpha}Q(\alpha^{\vee})(y_{\alpha}|-)\\
        =&\sum_{\alpha\in \Delta}k_{\alpha}Q(\alpha^{\vee})\omega_{\alpha}.
    \end{align*}
\end{proof}

Sometime it is useful to rewrite the expression in Lemma \ref{SimpTorChar} using the following identity.

\begin{lemma}\label{BQtoCP}
    Let $\alpha\in\Phi$ and $y\in Y$. Then $B_{Q}(y,\alpha^{\vee})=Q(\alpha^{\vee})\langle\alpha,y\rangle$.
\end{lemma}

\begin{proof}
    This follows from the $W$-invariance $B_{Q}(w_{\alpha}y,w_{\alpha}\alpha^{\vee})=B_{Q}(y,\alpha^{\vee})$.
\end{proof}

In particular, Lemmas \ref{SimpTorChar} and \ref{BQtoCP} imply the following corollary.

\begin{corollary}\label{TrivOnModCoRoots}
    Let $\chi\in \mathscr{Z}$ and $\alpha\in \Phi$. Then for any $u\in \mathfrak{o}^{\times}$
    \begin{equation*}
        \chi(\overline{h}_{\alpha}(u)^{n_{\alpha}})=1.
    \end{equation*}
\end{corollary}

\subsection{Parahoric subgroups}\label{ParaSubs}

Let $\mathcal{B}(G)=\mathcal{B}(\bf{G},F)$ be the enlarged Bruhat-Tits building of $G$. For any $x\in\mathcal{B}(G)$ there is a parahoric subgroup $P_{x}$ \cite{MP94} with maximal pro-$p$ subgroup $P_{x}^{+}$. The quotient $L=L_{x}:=P_{x}/P_{x}^{+}$ is isomorphic to the $\mathfrak{f}$-points of a connected reductive group $\mathbf{L}=\mathbf{L}_{x}$. Let $W_{x}$ be the Weyl group of $\mathbf{L}_{x}$. If we write $P=P_{x}$, then we may write $W_{P}=W_{x}$. 

For the rest of this subsection we assume $x$ is an element of the apartment of $T$, which we identify with $Y\otimes \mathbb{R}$ via the pinning. With this identification let $\mathscr{A}_{0}$ be the unique alcove contained in the positive Weyl chamber with respect to $\Phi^{+}$ and containing $0$ in its closure. 

The Weyl group $W_{P}=W_{x}$ can be naturally identified with the stabilizer of $x$ in $W_{\mathrm{af}}$. The root system of $W_{x}$ can be realized as
\begin{equation*}
    \Phi_{x}:=\{\alpha\in \Phi|\alpha(x)\in \mathbb{Z}\}.
\end{equation*}

If $x\in Y\otimes \mathbb{R}$ is an element in an alcove, then $P_{x}$ is an Iwahori subgroup. If $x\in \mathscr{A}_{0}$ we write $I:=P_{x}$.

Next we recall the theory of parahoric subgroups of the covering group $\overline{G}$ \cite{HW09,W11}. In particular, we want to understand the structure of $\overline{L}:=\overline{P}/P^{+}$. Note that $\overline{L}$ is well defined since $\mathrm{GCD}(n,p)=1$ implies $P^{+}$ has a unique splitting \cite[Proposition 2.3]{HW09}.

While $L$ is the points of an algebraic group, this may not be true for $\overline{L}$; though $\overline{L}$ is the central quotient of a reductive group of points. This is where the group $\overline{G}^{u}$ comes in. Now we make this precise.

We use the Hilbert symbol to identify $\mathfrak{f}^{\times}$ and $\mu_{q-1}\subset F^{\times}$. Specifically we have the isomorphism $\mathfrak{f}^{\times}\rightarrow \mu_{q-1}$ defined by $u\mapsto (\varpi,u^{\prime})_{q-1}$, where $u^{\prime}\in \mathfrak{o}^{\times}$ is any lift of $u$.

Brylinski-Deligne \cite[Construction 12.11]{BD01} show that the central extension
\begin{equation}\label{ParaCExt}
\mathfrak{f}^{\times}\hookrightarrow \overline{P}^{u}/P^{+}\twoheadrightarrow P/P^{+}
\end{equation}
arises as the $\mathfrak{f}$-points of a central extension of connected reductive groups
\begin{equation*}
\mathbf{G}_m\hookrightarrow \overline{\mathbf{L}}^{u}\twoheadrightarrow \mathbf{L}.
\end{equation*}
The root datum of the group $\overline{\mathbf{L}}^{u}$ is determined in \cite{W11} and it plays an important role in the definition of $\Phi_{\chi,\mathrm{af}}$ on line \eqref{DefPhiChiAf}.

The group $\overline{L}$ arises as the push out of the central extension \eqref{ParaCExt} via the $\frac{q-1}{n}$-th power map $\mu_{q-1}\rightarrow \mu_{n}$. In particular, $\overline{L}$ is a quotient of $\overline{L}^{u}:=\overline{\mathbf{L}}^{u}(\mathfrak{f})$ by the central subgroup $\mu_{\frac{q-1}{n}}$. We abuse notation and write this quotient map as $\overline{\wp}:\overline{L}^{u}\twoheadrightarrow \overline{L}$.

%%%%%%%%%%%%%%%%%%%%%%%%%%%%%%%%%%%%%%%%%%%%%%%
%%%%%%%%%%%%%%%%%%%%%%%%%%%%%%%%%%%%%%%%%%%%%%%
%Hecke Algebras
%%%%%%%%%%%%%%%%%%%%%%%%%%%%%%%%%%%%%%%%%%%%%%%
%%%%%%%%%%%%%%%%%%%%%%%%%%%%%%%%%%%%%%%%%%%%%%%

\subsection{Hecke Algebras}

In this subsection we intoduce several Hecke algebras. 

We begin with abstract Hecke algebras. See \cite[Section 7.1]{H90} for more details.
\begin{definition}
    Let $(W^{\prime},S^{\prime})$ be a Coxeter group and let $\ell^{\prime}$ be its length function. Let $\mathscr{H}(W^{\prime},S^{\prime})$ be a $\mathbb{C}$-vector space with standard basis elements $t_{w}$ indexed by $W^{\prime}$. The algebra structure is determined by the following relations.
    
    \begin{enumerate}
    \item $t_{1}$ is the unit element.
    \item $t_{s}t_{w}=t_{sw}$, if $s\in S^{\prime}$, $w\in W^{\prime}$ and $\ell^{\prime}(sw)=\ell^{\prime}(w)+1$.
    \item $t_{s}^{2}=(q-1)t_{s}+qt_{1}$, where $s\in S^{\prime}$.
    \end{enumerate}
\end{definition}

Recall that $\mathscr{H}(W^{\prime},S^{\prime})$ always has the following two characters. The index character is the unique $\mathbb{C}$-algebra homomorphism $\mathrm{ind}:\mathscr{H}(W^{\prime},S^{\prime})\rightarrow \mathbb{C}$ such that for $s\in S$, $\mathrm{ind}(t_{s})=q$. The sign character is the unique $\mathbb{C}$-algebra homomorphism $\mathrm{sgn}:\mathscr{H}(W^{\prime},S^{\prime})\rightarrow \mathbb{C}$ such that for $s\in S$, $\mathrm{sgn}(t_{s})=-1$.

Next we introduce the Hecke algebras coming from $\overline{G}$. 
\begin{definition}
    Let
    \begin{equation*}
    \mathcal{H}:=C^{\infty}_{c,\varepsilon}(I^{+}\backslash \overline{G}/I^{+})
    \end{equation*}
    be the $\mathbb{C}$-vector space of $\varepsilon^{-1}$-genuine locally constant compactly supported functions $f:\overline{G}\rightarrow \mathbb{C}$ that are $I^{+}$-biinvariant. For $g\in \overline{G}$ define $T_{g}$ to be the unique element in $\mathcal{H}$ such that $\mathrm{Supp}(T_{g})=\mu_{n}I^{+}gI^{+}$ and $T_{g}(g)=1$. There is a convolution product $*:\mathcal{H}\otimes \mathcal{H}\rightarrow \mathcal{H}$ defined by 
    \begin{equation*}
    f_{1}*f_{2}(g):=\int_{\overline{G}}f_{1}(h)f_{2}(h^{-1}g)dh,
    \end{equation*}
    where $dh$ is a Haar measure on $\overline{G}$ normalized so that the measure of $I^{+}$ is equal to $1$. With this product $\mathcal{H}$ is an associative $\mathbb{C}$-algebra with identity element $T_{1}$. The structure of $\mathcal{H}$ is studied in \cite{GGK}.
    
    Let $P\subset G$ be a parahoric subgroup containing $I$ and let $P^{+}$ denote its pro-unipotent radical. Define 
    $$\mathcal{H}_{P}:=C^{\infty}_{c,\varepsilon}(I^{+}\backslash \overline{P}/I^{+}).$$
\end{definition}

The inclusion $\overline{T}_{\mathfrak{o}}\subset \overline{I}$ induces an isomorphism $\overline{T}_{\mathfrak{o}}/T_{\mathfrak{o}}^{+}\simeq\overline{I}/I^{+}$. This induces an isomorphsim of $\mathbb{C}$-algebras $C^{\infty}_{\varepsilon}(I^{+}\backslash \overline{I}/I^{+})=C^{\infty}_{\varepsilon}(\overline{I}/I^{+})\simeq C^{\infty}_{\varepsilon}(\overline{T}_{\mathfrak{o}}/T_{\mathfrak{o}}^{+})$.

\begin{definition}
    Let $\chi\in \mathrm{Hom}(\overline{T}_{\mathfrak{o}}/T_{\mathfrak{o}}^{+},\mathbb{C}^{\times})$ be an $\varepsilon$-genuine character. Then we write $e_{\chi}\in C^{\infty}_{\varepsilon}(I^{+}\backslash \overline{I}/I^{+})$ for the element corresponding to $\chi^{-1}\in C^{\infty}_{\varepsilon}(\overline{T}_{\mathfrak{o}}/T_{\mathfrak{o}}^{+})$ normalized so that $e_{\chi}$ is an idempotent. 
    Define
    $$\mathcal{H}_{\chi}:=e_{\chi}\mathcal{H}e_{\chi}.$$
    Let  $W_{\chi}:=\mathrm{Stab}_{W_{\mathrm{ea}}}(\chi)$,
    \begin{equation}\label{DefPhiChiAf}
        \Phi_{\chi,\mathrm{af}}:=\{\alpha+k\in \Phi_{\mathrm{af}}|\chi((\varpi,u)^{kQ(\alpha^{\vee})}\overline{h}_{\alpha}(u))=1\text{ for all }u\in \mathfrak{o}^{\times}\},
    \end{equation}
    and let $W_{\chi}^{0}$ be the subgroup of $W_{\mathrm{af}}$ generated by the reflections $s_{a}$, where $a\in \Phi_{\chi,\mathrm{af}}$. 
    
    Let $\mathscr{T}:N(\overline{T})\twoheadrightarrow W_{\mathrm{ea}}$ to be the composition of the group homomorphisms $\wp:N(\overline{T})\rightarrow N(T)$ and the Tits homomorphism $N(T)\rightarrow W_{\mathrm{ea}}$ \cite[\S1.2]{T79}. Let $N^{0}_{\chi}:=\mathscr{T}^{-1}(W_{\chi}^{0})$.
    Define 
    $$\mathscr{H}_{\chi}\subset \mathcal{H}_{\chi}$$
    to be the subspace spanned by elements of the form $T_{n}$, where $n\in N_{\chi}^{0}$. 
\end{definition}

Most of the results we need about $\mathcal{H}$ and $\mathcal{H}_{\chi}$ are in \cite{GGK,Wang24}. We note a notational difference in these two references. Specifically, \cite{GGK} uses $\varepsilon$-genuine Hecke algebras, while \cite{Wang24} uses $\varepsilon^{-1}$-genuine Hecke algebras. In this paper we use the later notation.

The next theorem is an equivalent restatement of \cite[Theorem 4.6]{Wang24}, altered to be more suitable for our computations.

Define $\mathscr{A}_{\chi,0}$ to be the unique connected component of $Y\otimes \mathbb{R}-\cup_{a\in \Phi_{\chi,af}}H_{a}$ containing $\mathscr{A}_{0}$. 
Let $\Delta_{\chi}$ be the set of simple affine roots with respect to $\mathscr{A}_{\chi,0}$. 
Let $S_{\chi}:=\{s_{a}|a\in \Delta_{\chi}\}$. % and let $\Omega_{\chi}=\{w\in W_{\chi}|w \mathscr{A}_{\chi,0}=\mathscr{A}_{\chi,0}\}$. 
%Then $(W_{\chi}^{0},S_{\chi})$ is an affine Weyl group. Define $\mathscr{H}(W_{\chi}^{0},S_{\chi})$.
Let $\ell_{\chi}$ be the length function of the Coxeter group $(W_{\chi}^{0},S_{\chi})$. For $\chi\in\mathscr{Z}$, let $\widehat{\chi}:N_{\chi}^{0}\rightarrow \mathbb{C}^{\times}$ be the character from \cite[Lemma 5.2]{Wang24} restricted to $N_{\chi}^{0}$.

\begin{theorem}{\cite[Theorem 4.6]{Wang24}}\label{IMPres}
    For every $w\in W_{\chi}^{0}$ fix a representative $\dot{w}\in N_{\chi}^{0}$ such that $\mathscr{T}(\dot{w})=w$.
    
    The linear map 
    $$\mathscr{H}(W_\chi^0,S_\chi)\rightarrow \mathscr{H}_{\chi}, \quad t_{w}\mapsto q^{\frac{\ell_{\chi}(w)-\ell(w)}{2}}\widehat{\chi}(\dot{w})e_{\chi}T_{\dot{w}}e_{\chi}$$
    is well defined and defines a $\mathbb{C}$-algebra isomorphism. 
\end{theorem}

%We use the injection in the previous theorem to identify $\mathscr{H}_{\chi}$ with its image in $\mathcal{H}_{\chi}$.

We end this subsection with a lemma for later use. If $w\in W$ and $\chi$ is a character of $\overline{T}_{\mathfrak{o}}$ define $^{w}\chi(t):=\chi(\dot{w}^{-1}t\dot{w})$, where $\dot{w}\in N_{\overline{G}}(\overline{T})$ is any representative of $w$.% that will be used in the proof of Proposition \ref{ThetaPmod}.

\begin{lemma}\label{chiHPorbit}
    For each $w\in W_{P}$ we fix an element $\dot{w}\in \overline{P}$ representing $w$.
    
    Let $\chi\in \mathrm{Hom}(\overline{T}_{\mathfrak{o}}/T_{\mathfrak{o}}^{+},\mathbb{C}^{\times})$ be an $\varepsilon$-genuine character. Then 
    \begin{equation*}
        \{e_{^{w}\chi}*T_{\dot{w}}*e_{\chi}|w\in W_{P}\}
    \end{equation*}
    is a basis for $\mathcal{H}_{P}e_{\chi}$. Thus $\mathrm{dim}\,\mathcal{H}_{P}e_{\chi}=|W_{P}|$.
\end{lemma}

\begin{proof}
    This follows from the Bruhat decomposition and the identity $T_{\dot{w}}*e_{\chi}=e_{^{w}\chi}*T_{\dot{w}}*e_{\chi}$ \cite[Lemma 4.2, (ii)]{GGK}. (In the notation of \cite{GGK}, $e_{\chi}$ is $c(\chi)$.)
\end{proof}

\subsection{Moy-Prasad filtration and the exponential map}\label{sec:moyprasad}
%Let $\mathcal B(G)$ denote the building for $G$.

In this section we define an exponential map to the covering group.

For $x\in \mathcal B(G)$ and $r\ge 0$, let $G_{x,r}$ and $\mathfrak g_{x,r}$ be the Moy-Prasad filtration subgroups as defined in \cite{moyprasad}.
Let 
$$G_{x,r^+} = \bigcup_{s>r}G_{x,s}, \quad \mathfrak g_{x,r^+} = \bigcup_{s>r}\mathfrak g_{x,s}$$
and
$$G_{r^+} = \bigcup_{x\in \mathcal B(G)}G_{x,r^+}, \quad \mathfrak g_{r^+} = \bigcup_{x\in\mathcal B(G)}\mathfrak g_{x,r^+}.$$

We require the following hypothesis on the residue characteristic of $F$ to define our exponential map.

\begin{hypothesis}\label{hyp:exp}
    Either
    \begin{enumerate}
        \item the exponential function $\exp$ converges on $\mathfrak g_{0^+}$, or 
        \item there is a suitable mock-exponential function defined on $\mathfrak g_{0^+}\to G_{0^+}$.
    \end{enumerate}
    We refer to \cite[Lemma 3.2]{barmoy} for a sufficient condition to guarantee the convergence of the exponential map. 
    We refer to \cite[Remark 3.2.2]{debacker} for a discussion on mock exponential maps.
    In either case we write
    $$\exp:\mathfrak g_{0^+}\to G_{0^+}$$
    for the chosen function.
    %The exponential function $\exp$ converges on $\mathfrak g_{0^+}$.
    %We refer to \cite[Lemma 3.2]{barmoy} for sufficient conditions for this to hold.
\end{hypothesis}

\begin{definition}
    Fix $r\ge 0$.
    Since $G_{x,r^+}$ is a pro-p group and $\mathrm{GCD}(n,p)=1$ it admits a unique splitting
    $$s_{x,r}:G_{x,r^+}\to \overline G.$$
    Define $\tilde G_{x,r^+}$ to be the image of $s_{x,r}$.
    Define 
    $$\tilde G_{r^+} = \cup_x \tilde G_{x,r^+}.$$
    By uniqueness of the splittings, the $s_{x,r}$ agree on overlaps, and so we can define an exponential map
    $$\exp:\mathfrak g_{r^+}\to \tilde G_{r^+}$$
    by patching the maps
    $$s_{x,r}\circ\exp:\mathfrak g_{x,r^+}\to \tilde G_{x,r^+}.$$

\end{definition}

\begin{definition}
    Let $(\pi,V)$ be an irreducible smooth representation of $\overline G$. 
    We say $\pi$ has depth $r$, if $r$ is minimal such that there exists an $x\in \mathcal B(G)$ such that $\pi^{\tilde G_{x,r^+}}\ne 0$.
\end{definition}

\subsection{DeBacker's lifting map and stable orbits}
\begin{definition}
    Let $\mathcal N$ denote the set of nilpotent elements of $\mathfrak g$.
    The group $\overline G$ acts by the adjoint action on $\mathcal N$ and we write $\mathcal N_o$ for the set of orbits under this action.
    \begin{enumerate}
        \item The action of $\overline G$ on $\mathfrak g$ factors surjectively through $G$ and so the set of nilpotent orbits of $\overline G$ and $G$ coincide.
        \item There are finitely many $G$-orbits of nilpotent elements.
        \item The set of nilpotent orbits have a partial order called the closure order defined by $\mathcal O_1\le \mathcal O_2$ if $\mathcal O_1\subseteq \overline{\mathcal O_2}$ where the closure is taken with respect to the Hausdorff topology.
    \end{enumerate}
    We say two nilpotent elements are stably conjugate if they are conjugate by an element in $\overline{\mathbf G}(\overline F)$.
    We write $\mathcal N_o^s$ for the set of stable equivalence classes of nilpotent orbits.
\end{definition}

We will require the following hypothesis on the residue characteristic of $F$ to hold.
\begin{hypothesis}\label{hyp:lift}
    The hypotheses in \cite[\S4.2]{debacker-nilpotent} hold.
    In particular that the residuce characteristic is good for $G_{x,0}/G_{x,0^+}$ for all $x$.
\end{hypothesis}

The following theorem of DeBacker allows us to relate nilpotent orbits of the reductive quotient of parahoric subgroups to nilpotent orbits of $G$.
\begin{theorem}
    \cite{debacker-nilpotent}
    Suppose hypothesis \ref{hyp:lift} holds.
    
    Let $x\in \mathcal B(G)$ and let
    $$\pi_x:\mathfrak g_{x,0} \to \mathfrak g_{x,0}/\mathfrak g_{x,0^+} = \mathrm{Lie}(G_{x,0}/G_{x,0^+})$$
    denote the quotient map.
    
    For any nilpotent orbit $\underline{\mathcal O}\subset \mathfrak g_{x,0}/\mathfrak g_{x,0^+}$ there is a unique minimal nilpotent orbit $\mathscr L_x(\underline{\mathcal O})$ of $\mathfrak g$ intersecting $\pi_x^{-1}(\underline{\mathcal O})$.
    
    Moreover, every nilpotent orbit of $\mathfrak g$ can be realised in this fashion for some choice of vertex $x$ in a fixed apartment.
\end{theorem}

The next lemma explains how DeBacker's lifting map works on stable orbits of split groups.
\begin{lemma}\label{lem:lifting}
    Suppose hypothesis \ref{hyp:lift} holds.
    
    Let $\mathbf G$ be a Chevalley group defined over $\mathbb Z$ with maximal torus $\mathbf T$.
    Let $F$ be a non-archimedean local field with residue field $\mathfrak f$.
    Let $G = \mathbf G(F)$, $T = \mathbf T(F)$.
    \begin{enumerate}
        \item The stable conjugacy classes of nilpotent orbits of $G$ are in bijection (via comparing weighted Dynkin diagrams) with the nilpotent orbits of $\mathbf G(\mathbb C)$.
        We denote the bijection by 
        $$\mathcal O\mapsto \mathcal O(\mathbb C).$$
        \item The bijection preserves the closure ordering with respect to the Zariski topology on both sides.
        \item Let $x$ be a point in the apartment for $T$. 
        There is a point $s\in \mathbf T(\mathbb C)$ such that the groups $G_{x,0}/G_{x,0^+}$ and $Z_{\mathbf G}(s)$ have the same root datum with respect to $\mathbf T(\mathfrak f)$ and $\mathbf T(\mathbb C)$ respectively.
        We write 
        $$\mathbb G := \mathbf G(\mathbb C), \quad \mathbb G_x:=Z_{\mathbf G(\mathbb C)}(s).$$
        \begin{enumerate}
            \item There is a bijection between the stable conjugacy classes of nilpotent orbits of $G_{x,0}/G_{x,0^+}$ (recall that $G_{x,0}/G_{x,0^+}$ is the $\mathfrak f$-points of split connected reductive group over $\mathfrak f$, so stable conjugacy means conjugate by an element of the $\overline {\mathfrak f}$-points of this group) and the nilpotent orbits of $Z_{\mathbf G(\mathbb C)}(s)$ which we denote by
            $$\underline{\mathcal O} \mapsto \underline{\mathcal O}(\mathbb C).$$
            The bijection preserves the (Zariski) closure ordering on both sides.
            \item If $\underline{\mathcal O}$ is a nilpotent orbit of $G_{x,0}/G_{x,0^+}$ then
            $$\mathscr L_x(\underline{\mathcal O})(\mathbb C) = \mathbb G.\underline {\mathcal O}(\mathbb C).$$
            To emphasize the $x$-dependence we will write 
            $$\mathrm{Sat}_{\mathbb G_x}^{\mathbb G}(\underline {\mathcal O}(\mathbb C)):=\mathbb G. \underline{\mathcal O}(\mathbb C).$$
        \end{enumerate}
    \end{enumerate}
\end{lemma}
\begin{proof}
    Parts $(1)$ and $(2)$ hold since $\overline F$ has characteristic zero and Dynkin's classification of nilpotent orbits, and the computations of the closure order is uniform across algebraically closed fields of characteristic 0.

    For part $(3)$ note that there is always a rational point $x'\in \mathcal B(G)$ such that $G_{x,0}=G_{x',0}$ and $G_{x,0^+}=G_{x',0^+}$.
    Thus the existence of $s$ follows from \cite[Theorem 4.1]{reeder-yu}.

    For $(3),(a)$ note that hypothesis \ref{hyp:lift} implies that the residue characteristic of $F$ is good for $G_{x,0}/G_{x,0^+}$ and so the existence of the bijection comes from \cite[Theorem 2.7]{premet}.
    The fact that the closure orderings agree follows from the tables in \cite{spaltenstein}.

    Part $(b)$ follows from the fact that the lifting map comes from a lifting map of $\mathfrak {sl}_2$-triples.
    In particular the middle elements of the triples have the same pairings with the character lattice of the split torus in which it is chosen to lie in and so they must have equal weighted Dynkin diagrams.
\end{proof}

\subsection{Character distributions and the local character expansion}
In this section we show that DeBacker's result about the existence and domain of validity of the local character expansion holds for the tame extensions under consideration in this paper.

\begin{hypothesis}\label{hyp:bilin}
    The Lie algebra $\mathfrak g$ admits a non-degenerate $G$-invariant symmetric bilinear form $B$.
    We refer to \cite[Proposition 4.1]{adler-roche} for the precise conditions on the residue characteristic for this to hold.
\end{hypothesis}

\begin{definition}
    For every nilpotent orbit $\mathcal O\in \mathcal N_o$ let $I_{\mathcal O}:C_c^\infty(\mathfrak g)\to \mathbb C$ denote the associated orbital integral. 
    It converges by \cite{rao} since $F$ has characteristic $0$.

    Suppose hypothesis \ref{hyp:bilin} holds.

    Let $\psi:F\to \mathbb C^\times$ be an additive character of $F$ trivial on $\mathfrak p$ and non-trivial on $\mathfrak o$.
    For a function $f\in C_c^\infty(\mathfrak g)$ we define the Fourier transform of $f$ to be
    $$\hat f(x):=\int_{\mathfrak g}\psi(B(x,y))f(y)dY.$$

    We write $\hat I_{\mathcal O}:C_c^\infty(\mathfrak g)\to \mathbb C$ for the Fourier transform of $I_{\mathcal O}$.
\end{definition}

\begin{theorem}
    \label{thm:local-char-expansion}
    Suppose hypotheses \ref{hyp:exp}, \ref{hyp:lift} and \ref{hyp:bilin} hold.
    
    Let $(\pi,V)$ be a smooth irreducible representation of $\overline G$ of depth-$r$.
    Then for all $f\in C_c^\infty(\tilde G_{r^+})$ we have
    $$\Theta_\pi(f) = \sum_{\mathcal O\in \mathcal N_o}c_{\mathcal O}(\pi)\hat I_{\mathcal O}(f\circ\exp).$$
\end{theorem}
\begin{proof}
    The main ingredient in Debacker's proof of this result for linear groups in \cite{debacker} is \cite[Lemma 3.2.2]{debacker}.
    The proof of this lemma relies on the result in \cite[\S7.2]{moyprasad} which states that if $\chi$ is a character of $\tilde G_{x,s}/\tilde G_{x,s^+}$ lying in $\pi^{\tilde G_{x,s^+}}$ and $s>r$ then $\chi$ is degenerate.
    In fact the proof in the linear case works for covering groups as well, and with this in hand, we can deduce that \cite[Remark 3.5.1]{debacker} stating $\hat \Theta_\pi\in \tilde J_{-r}$ (c.f. \cite{debacker} for the notation) holds.
    
    Once the problem is tranferred to the Lie algebra, since the adjoint action of $\overline G$ on $\mathfrak g$ factors surjectively through $G$, the notions of $G$ and $\overline G$-equivariance coincide and so the rest of the proof follows verbatim.
\end{proof}

\subsection{The wave front set and test functions}
\begin{definition}
    Let $(\pi,V)$ be a smooth admissible representation of $G$ of depth $r$.
    By theorem \ref{thm:local-char-expansion} there exists unique constants $c_{\mathcal O}(\pi)\in \mathbb C$ for all $\mathcal O\in \mathcal N_o$ such that for all $f\in C_c^\infty(\tilde G_{r^+})$ we have
    $$\Theta_\pi(f) = \sum_{\mathcal O\in \mathcal N_o}c_{\mathcal O}(\pi)\hat I_{\mathcal O}(f\circ\exp).$$
    The wave front set of $\pi$ is defined to be 
    $$\WF(\pi) = \max\{\mathcal O:c_{\mathcal O}(\pi)\ne 0\}$$
    where the maximum is with respect to the closure ordering. 
    
    The \emph{stable wave front set} $\WF^s(\pi)$ is defined to be the largest stable orbits (with respect to the \emph{Zariski} closure) that contain a nilpotent orbit $\mathcal O$ such that $c_{\mathcal O}(\pi)\ne 0$.
\end{definition}

In order to compute the wave front set, we wish to apply the ideas of \cite{barmoy} .
However there is an obstruction to applying their ideas and that is that in general $\overline G_{x,0}/\tilde G_{x,0^+}$ may be a non-linear cover of a reductive group over $\mathbb F_q$.
However, as noted in section \ref{ParaSubs}, for $\overline G^{u}$ the universal tame extension of $G$, it is a \emph{reductive} central $\mathbb G_m$-extension of $G_{x,0}/G_{x,0^+}$.
Since it is reductive we can define the same test functions used in \cite{barmoy}.
In the same way as in \cite{okada} we then obtain the following definitions and results.
\begin{definition}\label{def:kwf}
    Let $\overline G^{u}$ denote the universal tame extension of $G$.
    
    For a nilpotent orbit $\underline{\mathcal O}$ of $\overline G_{x,0}^{u}/G_{x,0^+}$ let $\Gamma_{x,\underline{\mathcal O}}$ denote the generalised Gelfand-Graev representation of $\overline G_{x,0}^{u}/G_{x,0^+}$ attached to $\underline{\mathcal O}$.

    Recall that the Kawanaka wave front set of a representation $\pi$ of $\overline G_{x,0}^{u}/G_{x,0^+}$ is defined to be the maximal stable orbit containing a nilpotent orbit $\underline{\mathcal O}$ such that $\langle \Gamma_{x,\underline{\mathcal O}},\pi\rangle \ne 0$.
    We denote the Kawanaka wave front set of $\pi$ by 
    $$\WF^s(\pi).$$
\end{definition}
\begin{theorem}\label{StableWFS}
    Suppose hypotheses \ref{hyp:exp}, \ref{hyp:lift} and \ref{hyp:bilin} hold.
    
    Let $\mathbf G$ be a Chevalley group defined over $\mathbb Z$.
    Let $G = \mathbf G(F)$ and let $\overline G^{u}$ be the universal tame extension of $G$.
    
    Let $\mathscr C$ denote a set of points in the building of $G$ so that every nilpotent orbit is in the image of $\mathscr L_x$ for some $x\in \mathscr C$.
    
    Let $(\pi,V)$ be a depth-$0$ smooth admissible representation of $\overline G^{u}$.
    Let $\underline \pi_{x}$ denote the $\overline{G}_{x,0}^{u}/G_{x,0^+}$-representation $\pi^{G_{x,0^+}}$.
    
    Then
    $$\WF(\pi) = \max\cup_{x\in \mathscr C}\{\mathscr L_{x}(\underline {\mathcal O}): \langle \Gamma_{x,\underline {\mathcal O}}, \underline{\pi}_x\rangle \ne 0\}.$$
    In particular $(-)(\mathbb C)$ of the stable wave front set can be calculated as
    $$\WF^s(\pi)(\mathbb C) = \max \{\mathrm{Sat}_{\mathbb G_x}^{\mathbb G}(\WF^s(\sigma)(\mathbb C)):\sigma\subset \underline{\pi}_x \text{ irreducible}, x\in \mathscr C\}.$$
\end{theorem}

The following lemma reduces the computation of the wave front set to the case of the universal tame extension.
\begin{lemma}\label{WFPullBack}
    Let $\overline G^u$ denote the universal tame extension of $G$ and $(\pi,V)$ be a representation of $\overline G$.
    Let $\pi^u$ denote the representation of $\overline G^u$ obtained by pulling $\pi$ back along $\overline G^u\to \overline G$.
    Then 
    $$\WF(\pi) = \WF(\pi^u).$$
\end{lemma}
\begin{proof}
    For the proof of this lemma we adopt the notation in Section \ref{sec:moyprasad}.
    
    Since $\pi^u$ is obtained from $\pi$ by inflating along $p:\overline G^u\to \overline G$, the distributions 
    $$\Theta_{\pi}:C_c^\infty(\tilde G_{r^+})\to \mathbb C$$
    and
    $$\Theta_{\pi^u}:C_c^\infty(\tilde G^u_{r^+})\to \mathbb C$$
    agree under the isomorphism
    $$C_c^\infty(\tilde G_{r^+})\xrightarrow{\sim}C_c^\infty(\tilde G^u_{r^+}), \quad f\mapsto f\circ p.$$
    It follows that the associated distributions on $C_c^\infty(\mathfrak g_{r^+})$ coincide. Hence they must have the same local character expansions, and so the same wave front sets.
\end{proof}

%%%%%%%%%%%%%%%%%%%%%%%%%%%%%%%%%%%%%%%%%%%%%%%
%%%%%%%%%%%%%%%%%%%%%%%%%%%%%%%%%%%%%%%%%%%%%%%
%Theta Representation
%%%%%%%%%%%%%%%%%%%%%%%%%%%%%%%%%%%%%%%%%%%%%%%
%%%%%%%%%%%%%%%%%%%%%%%%%%%%%%%%%%%%%%%%%%%%%%%

\section{Theta Representations}\label{Theta}

In this section we define the theta representations of $\overline{G}$ and reduce the computation of its wave front set to a combinatorial problem, which is solved in Section \ref{WFSComp}.

We assume hypotheses \ref{hyp:exp}, \ref{hyp:lift} and \ref{hyp:bilin} throughout, so we can apply Theorem \ref{StableWFS}.

\subsection{Basic properties of theta representations}\label{ThetaPropSS}

For $\nu\in X\otimes \mathbb{R}$ there is a unique character $\delta_{\nu}:T\rightarrow\mathbb{R}_{>0}$ such that for $y\in Y$ and $a\in F^{\times}$
\begin{equation*}
\delta_{\nu}(y\otimes a)=|a|^{\nu(y)}.
\end{equation*}

\begin{definition}
    \label{def:exceptional}
    A character $\nu\in X\otimes \mathbb{R}$ is called exceptional if $\nu(\alpha_{Q,n}^{\vee})=1$ for all $\alpha\in \Delta$. 
\end{definition}

Let $\nu\in X\otimes \mathbb{R}$ be an exceptional character. The theta representation $\Theta(\pi^{\dagger},\nu)$ of $\overline{G}$ attached to $\pi^{\dagger}$ and exceptional character $\nu$ is the unique Langlands quotient of the (normalized) principal series $\mathrm{Ind}_{\overline{B}}^{\overline{G}}(\pi^{\dagger}\otimes \delta_{\nu})$. (The Langlands quotient theorem for covers was proved in \cite{BJ13} and is applicable because $\nu$ is exceptional.)
When there is no possibility for confusion we may write $\Theta$ for $\Theta(\pi^{\dagger},\nu)$. We may also write $\Theta(\chi^{\dagger},\nu)$, where $\chi^{\dagger}:Z(\overline{T})\rightarrow \mathbb{C}^{\times}$ is the distinguished genuine central character of $\pi^{\dagger}$.

The next proposition collects some properties of $\Theta$; item \eqref{TProp6} is crucial for the computation of the wave front set.

\begin{proposition}\label{ThetaBasics} $ $ Let $\chi\in \mathscr{Z}$. We also write $\chi$ for the inflation of $\chi$ to $\overline{I}$, which is trivial on $I^{+}$.
\begin{enumerate}
\item\label{TProp1} As $\overline{T}$-representations, $\Theta_{U}\cong \pi^{\dagger}\otimes \delta_{B}^{1/2}\cdot \delta_{w_{\ell}\cdot \nu}$, where $\Theta_{U}$ is the Jacquet module of $\Theta$ with respect to $U$.
\item\label{TProp2} The natural quotient map $\Theta\twoheadrightarrow \Theta_{U}$ induces an isomorphism of vector spaces, $\Theta^{I^{+}}\cong (\Theta_{U})^{T_{\mathfrak{o}}^{+}}=\Theta_{U}$. In particular, $\mathrm{dim}(\Theta^{I^{+}})=|\mathscr{X}|=|\mathscr{Z}|$.
\item\label{TProp3} Similarly, as vector spaces $\Theta^{(\overline{I},\chi)}\cong (\Theta_{U})^{(\overline{T}_{\mathfrak{o}},\chi)}$. In particular, $\mathrm{dim}(\Theta^{(\overline{I},\chi)})=1$.
\item\label{TProp4} As $\overline{I}$-modules, $\Theta^{(\overline{I},\chi)}\cong \chi$.
\item\label{TProp5} As $\overline{I}$-modules, $\Theta^{I^{+}}\cong \oplus\chi^{\prime}$, where the sum is over $\mathscr{Z}$.
\item\label{TProp6} As an $\mathscr{H}_{\chi}\subset \mathcal{H}_{\chi}$ module, $\Theta^{(\overline{I},\chi)}$ is isomorphic to the index character.
\end{enumerate}
\end{proposition}

\begin{proof}
(\ref{TProp1}) follows from \cite{Gao20}, Theorem 3.6, (i).

(\ref{TProp2}) follows from \cite{GGK}, Corollary 3.5, and Bushnell-Kutzko type theory \cite{BK98}, which is based on an argument that goes back to Borel-Casselman.

(\ref{TProp3}) is the same as (\ref{TProp2}).

(\ref{TProp4}) and (\ref{TProp5}) follow from (\ref{TProp1}) along with (\ref{TProp3}) and (\ref{TProp2}) respectively. 

(\ref{TProp6}) follows by combining (\ref{TProp1}) and (\ref{TProp3}). We add a few more details.

Let $f:\Theta^{(\overline{I},\chi)} \rightarrow \Theta_{U}^{(\overline{T}_{\mathfrak{o}},\chi)}$ be the vector space isomorphism of (\ref{TProp3}). Let $y\in Y_{Q,n}^{sc}$. Define $\theta_{\mathbf{s}_{y}}\in\mathcal{H}$ to be the element denoted by $\Theta_{\mathbf{s}_{y}}$ in \cite[Section 4.1]{GGK}.

The map $f$ intertwines the action of the element $\mathbf{s}_{y}$ on $\Theta_{U}^{(\overline{T}_{\mathfrak{o}},\chi)}$ with the action of the element $\theta_{\mathbf{s}_{y}}$ on $\Theta^{(\overline{I},\chi)}$. Specifically, for any $v\in\Theta^{(\overline{I},\chi)}$ a direct calculation shows

\begin{equation*}
    f(\theta_{\mathbf{s}_{y}}e_{\chi}\cdot v)=q^{\langle\rho,y\rangle}\mathbf{s}_{y}\cdot f(v).
\end{equation*}

By \eqref{TProp1}, $\mathbf{s}_{y}\cdot f(v)=q^{\langle\rho_{Q,n}-\rho,y\rangle}\chi^{\dagger}(\mathbf{s}_{y})f(v)$, thus
\begin{equation*}
    \chi^{\dagger}(\mathbf{s}_{y})^{-1}\theta_{\mathbf{s}_{y}}e_{\chi}\cdot v=q^{\langle\rho_{Q,n},y\rangle}v.
\end{equation*}

%By \cite[Lemma 4.14]{GGK}, we see that the $\mathrm{span}\{\chi^{\dagger}(\mathbf{s}_{y})^{-1}\theta_{\mathbf{s}_{y}}e_{\chi}|y\in Y_{Q,n}\}$ is a subalgebra of $\mathcal{H}_{\chi}$ isomorphic to $\mathbb{C}[Y_{Q,n}]$, and in fact this is Bernstein's subalgebra of $\mathcal{H}_{\chi}$.

The subalgebra $\mathscr{H}_{\chi}$ has an Iwahori-Matsumoto presentation, described in Theorem \ref{IMPres}. We restrict the action of $\mathcal{H}_{\chi}$ to $\mathscr{H}_{\chi}$. Each generator $t_{s_{a}}$, where $s_{a}\in \Delta_{\chi}$, can only act by the value $-1$ or $q$.

Let $y\in Y_{Q,n}^{sc}\cap W_{\chi}^{0}$ be dominant with respect to $\Delta$. Then $q^{\langle\rho_{Q,n},y\rangle}\chi^{\dagger}(\mathbf{s}_{y})^{-1}\theta_{\mathbf{s}_{y}}e_{\chi}=t_{y}$. Thus $t_{y}$ acts by $q^{\langle2\rho_{Q,n},y\rangle}$. 

Now $\ell_{\chi}(y)=\langle2\rho_{Q,n},y\rangle$ by \cite[Proposition 1.23]{IM}, \cite[Proposition 4.2]{Wang24}, and Corollary \ref{TrivOnModCoRoots}. Thus each $t_{s_{a}}$, where $s_{a}\in \Delta_{\chi}$, must act by $q$, as desired.
\end{proof}

\begin{comment}
\begin{remark}
    We note that if \cite[Theorem 1.1]{BM89} generalizes to covering groups, then Proposition \ref{ThetaBasics}, (\ref{TProp6}) and implies that $\Theta$ is a unitary representation.
\end{remark}
\end{comment}

Let $P\subset G$ be a parahoric subgroup containing $I$. Our next objective is to describe $\Theta^{P^{+}}$ as a $\overline{P}/P^{+}$-module.

\begin{proposition}\label{ThetaPmod}
    As $\overline{P}/P^{+}$-modules
    \begin{equation*}
        \Theta^{P^{+}}\simeq \oplus_{\mathcal{O}} \sigma_{\mathcal{O}},
    \end{equation*}
    where we have the following.
    \begin{itemize}
        \item The sum is over the $W_{P}$-orbits $\mathcal{O}\subset\mathscr{Z}$.
        \item Each $\sigma_{\mathcal{O}}$ is an irreducible $\overline{P}/P^{+}$-module.
        \item As a $\overline{T}_{\mathfrak{o}}$-module, $\sigma_{\mathcal{O}}^{I^{+}}\simeq \oplus_{\chi^{\prime}\in \mathcal{O}}\chi^{\prime}$.
    \end{itemize}

    In particular, we note that $\Theta^{P^{+}}$ is multiplicity free as a $\overline{P}/P^{+}$-module.
\end{proposition}

\begin{proof}
    Proposition \ref{ThetaBasics}, \eqref{TProp5} implies the following two properties of $\Theta$. First, $0\neq \Theta^{I^{+}}\subset \Theta^{P^{+}}$. Second, the representations $(\overline{I},\chi)$, where $\chi\in \mathscr{Z}$ are depth $0$ types of $\Theta$.
    
    Let $\sigma$ be an irreducible constituent of $\Theta^{P^{+}}$. By depth $0$ Moy-Prasad theory for covering groups \cite[Propositions 3.5,3.6]{HW09}, there exists $\chi\in\mathscr{Z}$ such that $\sigma^{(\overline{I},\chi)}$ is a one dimensional space. Let $v_{\chi}\in \sigma^{(\overline{I},\chi)}$ be nonzero.
    
    Let $\mathcal{O}\subset \mathscr{Z}$ be the $W_{P}$-orbit of $\chi$. Now since $\sigma^{I^{+}}$ is an irreducible $\mathcal{H}_{P}$-module we have a surjective map of $\mathcal{H}_{P}$-modules $\mathcal{H}_{P}e_{\chi}\twoheadrightarrow \sigma^{I^{+}}$, defined by $he_{\chi}\mapsto he_{\chi}\cdot v_{\chi}$. By Lemma \ref{chiHPorbit}, $\sigma^{I^{+}}\simeq \oplus_{\chi^{\prime}\in\mathcal{O}}\chi^{\prime}$ as $\overline{T}_{\mathfrak{o}}$-modules .
    
    Finally every $W_{P}$-orbit $\mathcal{O}\subseteq \mathscr{Z}$ must appear exactly once by Proposition \ref{ThetaBasics}, (\ref{TProp5}) and the inclusion $\Theta^{I^{+}}\subseteq \Theta^{P^{+}}$. Thus the result follows.
\end{proof}

%%%%%%%%%%%%%%%%%%%%%%%%%%%%%%%%%%%%%%%%%%%%%%%
%%%%%%%%%%%%%%%%%%%%%%%%%%%%%%%%%%%%%%%%%%%%%%%
%Kawanaka Wave Front Set
%%%%%%%%%%%%%%%%%%%%%%%%%%%%%%%%%%%%%%%%%%%%%%%
%%%%%%%%%%%%%%%%%%%%%%%%%%%%%%%%%%%%%%%%%%%%%%%

\subsection{The Kawanaka wave front set of semisimple principal series representations}\label{KWFSGen}
In this subsection only we abandon out previous notation. Let $\mathbb G$ denote a connected reductive group defined over $\mathbb F_q$ and write $G=\mathbb G(\mathbb F_q)$ for its $\mathbb F_q$-points.
\begin{hypothesis}
    \label{hyp:wf}
    Suppose the characteristic of $\mathbb F_q$ is good for $\mathbb G$ (c.f. \cite[\S2.1]{premet} for the definition of a good prime).
\end{hypothesis}

In this section we will determine an explicit formula for the Kawanaka wave front set (see definition \ref{def:kwf}) for semisimple principal series representations.
We refer to \cite[Definition 2.6.9]{Geck_Malle_2020} for the definition of a semisimple character.

\begin{proposition}\cite[Theorem 3.1.28]{Geck_Malle_2020}
    Let $B$ be a Borel subgroup of $G$ and $T$ be a maximal torus in $B$ with Weyl group $W$.
    Let $\chi$ be a character of $T$ and $\mathrm{Stab}_W(\chi)$ denote the stabiliser of $\chi$ in $W$.
    There is a natural identification
    \begin{equation}
        \label{eq:hecke}
        \End(\mathrm{Ind}_{B}^{G}(\chi))\cong C((B,\chi)\backslash G / (B,\chi)) =: \mathcal H
    \end{equation}
    where the latter is a convolution algebra after fixing a choice of Haar measure.
    \begin{enumerate}
        \item Any irreducible constituent $\sigma$ of $\mathrm{Ind}_B^G(\chi)$ gives rise to an $\mathcal H$-module $\sigma^{(B,\chi)}$ via
        \begin{equation}
            \label{eq:h-module}
            f\star v = \int_Gf(g)\sigma(g)vdg, \quad f\in \mathcal H,v\in \sigma^{(B,\chi)}.
        \end{equation}
        \item Any function in $\mathcal H$ is supported on $B\mathrm{Stab}_W(\chi)B$ and the algebra decomposes as a direct sum of vector spaces
        $$\bigoplus_{w\in \mathrm{Stab}_W(\chi)} \mathcal H_w$$
        where $\mathcal H_w$ consists of functions in $\mathcal H$ supported on $BwB$.
        \item Let $\Phi_\chi$ denote the subroot system of $\Phi$ consisting of roots $\alpha$ such that $\chi\circ \alpha^\vee:\mathbb F_q^\times\to \mathbb F_q^\times$ is trivial and let $W_\chi$ denote the associated Weyl group.
        Let 
        $$\mathscr H = \bigoplus_{w\in W_\chi}\mathcal H_w.$$
        Then $\mathscr H$ is a subalgebra and admits a unique basis $\{T_w\in \mathcal H_w:w\in W_\chi\}$ with respect to which it admits the presentation of a Hecke algebra $\mathcal H(W_\chi,q)$.
    \end{enumerate}
\end{proposition}

\begin{proposition}\label{prop:semisimple}
    An irreducible subrepresentation $\sigma$ of $\mathrm{Ind}_B^G(\chi)$ is a semisimple character if the restriction of $\sigma^{(B,\chi)}$ to $\mathscr H$ is the index character of $\mathscr H$.
\end{proposition}
\begin{proof}
    Suppose first that $G$ has connected center. 
    Then $\mathscr H = \mathcal H$ and there is a unique semisimple character in the Lusztig series containing $\mathrm{Ind}_B^G(\chi)$. 
    By \cite[Proposition 8.3.7]{Carter}, $\sigma$ is semisimple if and only if $D_G(\sigma)$ is regular, where $D_G$ denotes Alvis-Curtis duality.
    But by \cite[Theorem 7.5]{HL83}, if $\sigma^{(B,\chi)}$ corresponds to the index character then $D_G(\sigma)^{(B,\chi)}$ corresponds to the sign character of $\mathscr H$. 
    By the argument at the end of \cite[\S7.2]{reeder}, it follows that $D_G(\sigma)$ must be regular and so we can indeed conclude that $\sigma$ is semisimple.

    Let us now drop the assumption that $G$ has connected center.
    Let $\phi:G\to \tilde G$ be a regular embedding \cite[\S1.7]{Geck_Malle_2020} and view $G$ as a subgroup of $\tilde G$ by identifying it with its image.
    In particular $\tilde G$ has connected center.
    Let $\tilde B$ be a Borel of $\tilde G$, and $\tilde T$ be a maximal torus of $\tilde G$ in $\tilde B$, all chosen so that $B\subset \tilde B,T\subset \tilde T$.
    Let $\tilde\chi$ be a character of $\tilde T$ such that $\chi = \tilde\chi|_G$.
    Let $\mu_G$ denote the Haar measure on $G$ and normalise the Haar measure $\mu_{\tilde G}$ on $\tilde G$ so that
    $$|\tilde G:G|\mu_{\tilde G}|_G = \mu_G.$$
    With this normalisation of Haar measures the map
    $$C((\tilde B,\tilde \chi)\backslash \tilde G/(\tilde B,\tilde \chi)) \to \mathcal H, \quad f\mapsto f|_G$$
    is an isomorphism onto $\mathscr H$.
    We will prove the following claim: if $\tilde \sigma$ is an irreducible constituent of $\mathrm{Ind}_{\tilde B}^{\tilde G}(\tilde\chi)$, and $\sigma$ is an irreducible constituent of $\tilde\sigma|_G$ then
    $$\Hom_{\mathscr H}(\tilde\sigma^{(\tilde B,\tilde\chi)},\sigma^{(B,\chi)})\ne0$$
    where the $\mathscr H$-module structure of $\tilde\sigma^{(\tilde B,\tilde\chi)}$ comes from $\mathscr H\simeq C((\tilde B,\tilde \chi)\backslash \tilde G/(\tilde B,\tilde \chi))$ and the $\mathscr H$-module structure on $\sigma^{(B,\chi)}$ is given by the restriction of the $\mathcal H$ action to $\mathscr H$.
    Once we have established this claim we are done because every $\chi$-principal series representation $\sigma$ lies in the restriction of a $\tilde\chi$-principal series representation $\tilde\sigma$.
    Thus if $\sigma^{(B,\chi)}$ restricted to $\mathscr H$ is the index character of $\mathscr H$, then by the claim, $\tilde\sigma^{(\tilde B,\tilde\chi)}$ must also be the index character of $\mathscr H$.
    Since $\tilde G$ has connected center, $\tilde\sigma$ is semisimple, and so by \cite[Corollary 2.6.18]{Geck_Malle_2020}, $\sigma$ is semisimple.

    It remains to verify the claim.
    Let $\tilde \sigma$ be a $\tilde\chi$-pricipal series representation of $\tilde G$ and let
    $$\tilde\sigma|_G = \sigma_1\oplus\cdots \oplus \sigma_k.$$
    By \cite[Theorem 1.7.15]{Geck_Malle_2020}, this decomposition is multiplicity free. 
    By Cliffords theorem, the $\sigma_{j}$ are all conjugate under the action of $\tilde G$ and hence are all $\chi$-principal series representations.
    Since $\tilde G=G\tilde T$, we can moreover deduce that the conjugation action on the $\sigma_i$ by $\tilde T$ is transitive.
    
    Let $\star_{\tilde G}$ denote the action of $\mathscr H$ on $\tilde \sigma^{(\tilde B,\tilde\chi)}$ coming from $\mathscr H\simeq C((\tilde B,\tilde \chi)\backslash \tilde G/(\tilde B,\tilde \chi))$, and let $\star_G$ denote the action of $\mathcal H$ on $\tilde\sigma^{(B,\chi)}$.
    Since every coset $\tilde G/G$ has a representative in $\tilde T$, we have that for $f\in C((\tilde B,\tilde \chi)\backslash \tilde G/(\tilde B,\tilde \chi))$ and $v\in \tilde\sigma^{(\tilde B,\tilde\chi)}$
    \begin{align*}
        f\star_{\tilde G} v &= \int_{\tilde G}f(\tilde g)\tilde\sigma(g)v d\mu_{\tilde G}(\tilde g) = \sum_{g_0\in \tilde G/G}\int_{G}f(gg_0)\tilde\sigma(gg_0)vd\mu_{\tilde G}(g) \\
        &= \sum_{g_0\in \tilde G/G}\int_{G}f(g)\tilde\chi(g_0)^{-1}\tilde\sigma(g)\tilde\chi(g_0)vd\mu_{\tilde G}(g) = |\tilde G:G|\int_{G}f(g)\tilde\sigma(g)vd\mu_{\tilde G}(g) \\
        &= \int_{G}f(g) \tilde\sigma(g)d\mu_G(g) = f\star_G v.
    \end{align*}
    Thus the inclusion map $\iota:\tilde\sigma^{(\tilde B,\tilde\chi)}\to \tilde\sigma^{(B,\chi)}$ is $\mathscr H$-equivariant.
    
    Let $\pi_i:\tilde\sigma\to \sigma_i$ denote the $G$-equivariant projection onto the $\sigma_i$ factor.
    Since it is $G$-equivariant it induces a $\mathscr H$-equivariant map $\pi_i:\tilde\sigma^{(B,\chi)}\to \sigma_i^{(B,\chi)}$.
    Thus 
    $$\pi_i\circ \iota\in \Hom_{\mathscr H}(\tilde\sigma^{(\tilde B,\tilde\chi)},\sigma_i^{(B,\chi)}).$$
    It remains to show that $\pi_i\circ \iota\ne 0$ for all $1\le i\le k$.
    
    For $\tilde t\in \tilde T$ let $m_{\tilde t}:\tilde\sigma\to \tilde\sigma$ denote the multiplication by $\tilde t$ map, and let $\tau_{\tilde t}\in \mathrm{Sym}(\{1,\dots,k\})$ denote the permutation satisfying 
    $$\sigma_{\tau_{\tilde t}(i)}(g)\simeq \hphantom{ }^{\tilde t}\sigma_i(g):=\sigma_i(\tilde t^{-1}g\tilde t).$$
    The map $m_{\tilde t}$ induces a $\tilde G$-equivariant isomorphism $m_{\tilde t}:\hphantom{ }^{\tilde t}\tilde \sigma\to \tilde \sigma$ and, since $\tilde\sigma|_G$ is multiplicity free, restricts to $G$-equivariant isomorphisms $m_{\tilde t,i}:\sigma_i\to \sigma_{\tau_{\tilde t}(i)}$ so that 
    $$m_{\tilde t} = \bigoplus_im_{\tilde t,i}.$$
    Since the space of $G$-homomorphisms from $\tilde \sigma\to \sigma_i$ is one dimensional, there is a constant $c(\tilde t,i)\in \mathbb C^\times$ such that
    $$\pi_{\tau_{\tilde t}(i)}\circ m_{\tilde t,i} = c(\tilde t,i)\pi_{i}.$$
    Now, let $v\in \tilde\sigma^{(\tilde B,\tilde\chi)}$ be non-zero and write $v=v_1+\cdots+v_k$ where $v_i\in \sigma_i$.
    Since $v\ne0$ there is an $i$ such that $\pi_i(v)=v_i\ne0$.
    But since $v\in\tilde\sigma^{(\tilde B,\tilde \chi)}$, we have $\tilde t.v = \tilde\chi(\tilde t).v$ and so
    $$\tilde\chi(\tilde t)v_{\tau_{\tilde t}(i)} = \tilde\chi(\tilde t)\pi_{\tau_{\tilde t}(i)}(v) = \pi_{\tau_{\tilde t}(i)}(\tilde t v) = c(\tilde t,i)v_{i}.$$
    Since $\tilde T$ acts transitively on the $\sigma_i$ we can deduce that $v_j\ne0$ for all $j\ne0$ and so $\pi_i\circ \iota$ is injective for all $i$, and in particular it is non-zero.
\end{proof}

\begin{definition}
    \cite[\S11.2, j-induction]{Carter} 
    \begin{enumerate}
        \item Let $V$ be a Euclidean space and $W\subset GL(V)$ be a finite subgroup generated by reflections.
        Let $U$ denote the orthogonal complement of $V^{W}$ in $V$.
        We say an irreducible representation $E$ of $W$ has property $(P)$ if there exists an $i$ such that $E$ occurs with multiplicity $1$ in $S^i(U^*)$ and multiplicity $0$ in $S^j(U^*)$ for all $0\le j\le i$.
        We call $i$ the fake degree of $E$.
        \item Let $W'\subset W$ be a finite subgroup generated by reflections.
        Let $U'$ denote the orthogonal complement of $V^{W'}$ in $V$.
        Let $E$ be an irreducible representation of $W'$ with property $(P)$ and fake degree $i$.
        Viewing $S^i(U'^*)$ as a subspace of $S^i(U^*)$ we write 
        $$j_{W'}^W(E)$$
        for the $W$-submodule of $S^i(U^*)$ generated by $E$. 
        It is irreducible and also satisfies property $(P)$ with fake degree $i$.
    \end{enumerate}
\end{definition}

\begin{lemma}\cite{macdonald}
    Let $V$ be a Euclidean space and $W$ be the Weyl group of a root system $\Phi\subset V$. 
    The sign representation $w\mapsto \det(w)$ of $W$ satisfies property $(P)$ and has fake degree $|\Phi|/2$.
\end{lemma}

\begin{definition}
    [Springer support] Let $l$ be a prime not equal to $p$. 
    
    Let $E$ be an irreducible representation of $W$, the Weyl group of $G$.
    
    The Springer support $\mathcal O_{\mathrm{Spr}}(E)$ of $E$ is the largest nilpotent orbit $\mathcal O$ of $\mathbb G(\overline{\mathbb F_q})$ such that $E$ is a constituent of $H^i(\mathcal B_n,\mathbb Q_l)$ for some $n\in \mathcal O$ and some $i\ge 0$, where $\mathcal B_n$ denotes the variety of Borels containing $n$. 
    
    There is a unique such nilpotent orbit, and it is precisely the nilpotent orbit attached to $E$ via the Springer correspondence.
\end{definition}

\begin{corollary}\label{KWFS}
    Suppose hypothesis \ref{hyp:wf} is in effect.
    
    Let $\sigma$ be an irreducible constituent of the principal series $\mathrm{Ind}_{B}^{G}(\chi)$ such that as an $\mathscr{H}$-module $\sigma^{(B,\chi)}$ is isomorphic to the index character. Then
    $$\WF^s(\sigma) = \mathcal {O}_{\mathrm{Spr}}(j_{W_{\chi}}^{W}(\mathrm{sgn})).$$
\end{corollary}
\begin{proof}
    Since hypothesis \ref{hyp:wf} is in effect we may apply the results of \cite{TAYLOR_2016}.
    
    By Lemma \ref{prop:semisimple} we know that $\sigma$ is semisimple.
    In particular it belongs to the trivial two-sided cell of $W_\chi$ (c.f. the last displayed equation of \cite[\S14.5]{TAYLOR_2016}).
    The result then follows from \cite[Theorem 14.10]{TAYLOR_2016}.
\end{proof}

\subsection{The Kawanaka wave front set of $\sigma_{\mathcal O}$}\label{KWFSComp}

We now reinstate all of our previous notation in use before the previous subsection.

In this subsection, we compute the Kawanaka wave front set of the representations $\sigma_{\mathcal{O}}$ appearing in Proposition \ref{ThetaPmod}, by applying Corollary \ref{KWFS}. To begin we must translate from the setting of parahoric subgroups to finite groups of Lie type.

Let $\mathcal{O}\subset \mathscr{Z}$ be a $W_{P}$-orbit. Let $\sigma_{\mathcal{O}}\subset \Theta^{P^{+}}$ be an irreducible representation of $\overline{L}=\overline{P}/P^{+}$ as in Proposition \ref{ThetaPmod} and let $\chi\in\mathcal{O}$. 
Using $\overline{\wp}:\overline{L}^{u}\rightarrow \overline{L}$ we inflate $\sigma_{\mathcal{O}}$ to a representation of $\overline{L}^{u}$, which we still call $\sigma_{\mathcal{O}}$. 
Recall from Subsection \ref{ParaSubs} that $\overline{L}^{u}$ is the $\mathfrak{f}$-points of a connected reductive group $\overline{\mathbf{L}}^{u}$.

Let $\underline{B}$ be the Borel subgroup of $\overline{L}^{u}$ such that $\overline{\wp}(\underline{B})=\overline{I}/P^{+}$. The character $\chi\in\mathcal{O}$ defines a character of $\overline{I}$ which is trivial on $I^{+}\supset P^{+}$, thus defines a character of $\underline{B}$. We also call this character $\chi$.

\begin{lemma}
    The $\overline{L}^{u}$-representation $\sigma_{\mathcal{O}}$ is a constituent of the principal series $\mathrm{Ind}_{\underline{B}}^{\overline{L}^{u}}(\chi)$.
\end{lemma}

\begin{proof}
    This follow from Proposition \ref{ThetaPmod} and Frobenius reciprocity.
\end{proof}

Next we need to determine the structure of the one dimensional space $\sigma_{\mathcal{O}}^{(\underline{B},\chi)}=\sigma_{\mathcal{O}}^{(\overline{I},\chi)}$ as a module of a Hecke algebra. 

Let $\mathcal{H}_{P,\chi}:=e_{\chi}\mathcal{H}_{P}e_{\chi}\subset\mathcal{H}_{\chi}$. 
This Hecke algebra acts on $\sigma_{\mathcal{O}}^{\overline{I},\chi}$, and so does the subalgebra $\mathscr{H}_{P,\chi}:=\mathscr{H}_{\chi}\cap \mathcal{H}_{P,\chi}$. 

\begin{proposition}\label{ThetaParaHecke}
    Let $\mathcal{O}\subset \mathscr{Z}$ be a $W_{P}$-orbit and let $\chi\in \mathcal{O}$. 
    
    Then the $\mathscr{H}_{P,\chi}$-module $\sigma_{\mathcal{O}}^{(\overline{I},\chi)}$ is isomorphic to the index character.
\end{proposition}
\begin{proof}
    %This will be done in two steps. First, we determine the module structure with respect to a Hecke algebra on $\overline{P}$; the essential input is Proposition \ref{ThetaBasics}, (\ref{TProp6}). Second, we translate the result to a Hecke algebra on $\overline{L}^{u}$; this is routine.
    %We begin step one. 
    
    Recall that $\mathscr{H}_{\chi}=\mathscr{H}(W_{\chi}^{0},S_{\chi})$, where $(W_{\chi}^{0},S_{\chi})$ is an affine Weyl group derived from the affine root system $\Phi_{\chi,\mathrm{af}}\subset \Phi_{\mathrm{af}}$. Given $y\in Y\otimes \mathbb{R}$, we can construct a sub-root system
    \begin{equation*}
        \Phi_{\chi,y}:=\{\alpha\in \Phi|\alpha-\alpha(y)\in \Phi_{\chi,af}\}.
    \end{equation*}
    
    We take $y\in \bar{\mathscr{A}}_{0}\subset \bar{\mathscr{A}}_{0,\chi}$. In this case, $\Phi_{\chi,y}$ is the root system of the finite type Weyl group $(W_{\chi,y}^{0},S_{\chi,y})$, where $S_{\chi,y}$ consists of the elements of $S_{\chi}$ that fix $y$. This defines a subalgebra $\mathscr{H}(W_{\chi,y}^{0},S_{\chi,y})\subset \mathscr{H}_{\chi}$. 
    
    Now take $P=P_{y}$. Since $y\in \bar{\mathscr{A}}_{0}$, $I\subset P$. Then $\mathscr{H}_{P,\chi}$ is equal to $\mathscr{H}_{\chi,y}$, viewing both algebras as subalgebras of $\mathscr{H}_{\chi}$. This follows because both algebras are spanned by the elements $t_{w}$, where $w\in W_{\chi,y}^{0}=W_{y}\cap W_{\chi}^{0}$.
    
    The proposition then follows directly from Proposition \ref{ThetaBasics}, (\ref{TProp6}) and Proposition \ref{ThetaPmod}.
    
\end{proof}

%Now we move to step two. 
The next lemma relates the Hecke algebras on $P$ to Hecke algebras on $\overline{L}^{u}$. The proof is routine and omitted.

\begin{lemma}
The surjective map $\overline{\wp}:\overline{L}^{u}\rightarrow \overline{L}$ induces an isomorphism of $\mathbb{C}$-algebras 
$$\overline{\wp}^{*}:\mathcal{H}_{P}\rightarrow \mathcal{H}_{\overline{L}^{u}}:=C^{\infty}_{c}(\overline{L}^{u}/(\mu_{q-1},\varepsilon\circ \overline{\wp}))$$
defined by 
\begin{equation*}
f\mapsto (g\mapsto f(\overline{\wp}(g))).
\end{equation*}
Furthermore, the idempotent $e_{\chi}\in \mathcal{H}_{P}$ is mapped to the idempotent of the character $\chi$ viewed as a character of $\underline{B}$. 
Thus 
$$\overline{\wp}^{*}:\mathcal{H}_{P,\chi}\rightarrow \mathcal{H}_{\overline{L}^{u},\chi}:=C^{\infty}_{c}((\underline{B},\chi)\backslash\overline{L}^{u}/(\underline{B},\chi))$$
is an isomorphism of $\mathbb{C}$-algebras.
\end{lemma}

Putting everything together and applying Corollary \ref{KWFS} gives the following corollary. Let 
$$\mathscr{H}_{\overline{L}^{u},\chi}:=\overline{\wp}^{*}(\mathscr{H}_{P,\chi}).$$

\begin{corollary}\label{ParaKWFS}
    Let $\mathcal{O}\subset\mathscr{Z}$ be a $W_{P}$-orbit and let $\sigma_{\mathcal{O}}\subset \Theta^{P^{+}}$ be an irreducible $\overline{L}^{u}$-module. 
    Then for any $\chi\in\mathcal{O}$, the $\mathscr{H}_{\overline{L}^{u},\chi}$-module $\sigma_{\mathcal{O}}^{(\underline{B},\chi)}$ is isomorphic to the index character. In particular, 
    $$\mathrm{WF}^{s}(\sigma_{\mathcal{O}})=\mathcal{O}_{\mathrm{Spr}}(j_{W_{\chi,y}^{0}}^{W_{y}}\mathrm{sgn}).$$
\end{corollary}
Note that hypothesis \ref{hyp:lift} implies that hypothesis \ref{hyp:wf} holds for $P/P^+$ and so we can indeed apply \ref{KWFS}.

\subsection{A combinatorial reduction}\label{CombRed}
We perform one more reduction. 
Specifically, we describe the root system of the Weyl group $W_{\chi,y}^{0}$ purely in terms of root theoretic data.

By the definition of $\Phi_{\chi,af}$,
\begin{equation}\label{ChiyRootSys}
    \Phi_{\chi,y}=\{\alpha\in\Phi_{y}| \chi((\varpi,u)^{-\alpha(y)Q(\alpha^{\vee})}\overline{h}_{\alpha}(u))=1\text{ for all }u\in \mathfrak{o}^{\times}\}.
\end{equation}
\begin{remark}
    The factor $(\varpi,u)^{\alpha(y)Q(\alpha^{\vee})}$ appears because in covering groups the coroots of the reductive quotients of parahoric subgroups can change as you move through different points of the building. This phenomenon is carefully studied in \cite{W11}.
\end{remark}

We can describe $\Phi_{\chi,y}$ in terms of root theoretic data using Lemmas \ref{SimpTorChar} and \ref{BQtoCP}.

\begin{proposition}\label{SimpPhiChi} Let $y\in Y\otimes \mathbb{Q}$  and let $\chi\in \mathscr{Z}$. Then there exists $y_{\chi}\in Y^{sc}\otimes \mathbb{Q}$ a coweight such that 
    \begin{equation*}
        \Phi_{y,\chi}=\{\alpha\in \Phi_{y}| Q(\alpha^{\vee})\alpha(y-y_{\chi})\in n\mathbb{Z}\}.
    \end{equation*}
\end{proposition}

Given $y\in Y\otimes \mathbb{R}$ we define the following subroot system of $\Phi$,
\begin{equation*}
    \Phi^{Q,n}_{y}:=\{\alpha\in \Phi_{y}|Q(\alpha^{\vee})\alpha(y)\in n\mathbb{Z}\}.
\end{equation*}

\begin{lemma}\label{SimpRootSys}
   Let $y\in Y\otimes \mathbb{R}$  and let $\chi\in \mathscr{Z}$. Then there exists $x\in Y^{sc}\otimes \mathbb{R}$ a coweight such that
   \begin{equation*}
       \Phi_{y,\chi}=\Phi^{Q,n}_{y+x}
   \end{equation*}
\end{lemma}

\begin{proof}
    For any coweight $\omega^{\prime}\in Y\otimes \mathbb{Q}$ we have $\Phi_{y}=\Phi_{y+\omega^{\prime}}$. Now apply Proposition \ref{SimpPhiChi}.
\end{proof}

The next proposition is the main result of this section. 
It reduces the problem of computing the wave front set to a combinatorial problem, which we solve in Section \ref{WFSComp}. The main inputs in the proof of this proposition are Corollary \ref{ParaKWFS} and Theorem \ref{StableWFS}.

\begin{proposition}\label{WFThetaRedtoComb}
    Let $\nu\in X\otimes \mathbb{R}$ be an exceptional character. Let $\pi^{\dagger}\in\mathrm{Irr}(\overline{T})$ be distinuished and $\varepsilon$-genuine. Let $\Theta=\Theta(\pi^{\dagger},\nu)$ be a Theta representation. 

    Then
    \begin{equation*}
       \mathrm{WF}^{s}(\Theta)(\mathbb C)=\max \{\mathrm{Sat}_{\mathbb G_y}^{\mathbb G}(\mathcal{O}_{\mathrm{Spr}}(j_{W(\Phi_{y}^{Q,n})}^{W(\Phi_{y})}\mathrm{sgn})(\mathbb C)): y\in \mathscr V\},
    \end{equation*}
    where $\mathscr{V}$ is the set of vertices in the apartment for $T$.
\end{proposition}

\begin{proof}
    By Corollary \ref{WFPullBack} we may pullback $\Theta$ back to $\overline{G}^{u}$ to compute the wave front set. Now we can apply Theorem \ref{StableWFS} taking $\mathscr{C}=\mathscr{V}$ and $\pi=\Theta$. By Theorem \ref{StableWFS},
    \begin{equation*}
        \mathrm{WF}^{s}(\Theta)=\max \{\mathbb G.\WF^s(\sigma)(\mathbb C):\sigma\subset \Theta^{P_{y}^{+}} \text{ irreducible}, y\in \mathscr V\}.
    \end{equation*}
    Here $\Theta^{P_{y}^{+}}$ is viewed as an $\overline{L}^{u}_{y}$-module via pullback.
    For $y\in \mathscr{V}$ and any irreducible $\overline{L}^{u}$-module $\sigma\subset \Theta^{P_{y}^{+}}$ we apply Corollary \ref{ParaKWFS} to get
    \begin{equation*}
        \mathrm{WF}^{s}(\sigma)(\mathbb{C})=\mathcal{O}_{\mathrm{Spr}}(j_{W_{\chi,y}^{0}}^{W_{y}}\mathrm{sgn}),
    \end{equation*}
    where $\chi\in\mathscr{Z}$ is a $T_{\mathfrak{o}}$-module constituent of $\sigma$. Note that $j_{W_{\chi,y}^{0}}^{W_{y}}\mathrm{sgn}$ only depends on the underlying root systems. Since $W_{\chi,y}^{0}$ is the Weyl group of the root system $\Phi_{y,\chi}$ we can apply Lemma \ref{SimpRootSys} to see that there exists a coweight $x\in Y\otimes \mathbb{R}$ such that $\Phi_{y,\chi}=\Phi_{y+x}^{Q,n}$.

    Thus 
    \begin{equation*}
        \mathrm{WF}^{s}(\Theta)(\mathbb C)=\max \{\mathrm{Sat}_{\mathbb G_y}^{\mathbb G}(\mathcal{O}_{\mathrm{Spr}}(j_{W(\Phi_{y}^{Q,n})}^{W(\Phi_{y})}\sgn)(\mathbb C)): y\in \mathscr V\}.
    \end{equation*}
\end{proof}

\begin{comment}
Note that $\chi_{0}(\mathbf{s}(y(u)^{n}))=1$ for any $y\in Y$ \cite[Section 6.4]{GG18}. Let $\{y_{j}\}$ be a $\mathbb{Z}$-basis for $Y$. Then there exists $k_{j}\in \mathbb{Z}$ such that 
\begin{equation*}
    \chi_{0}(\mathbf{s}(y_{j}(u)))=(\varpi,u)^{k_{j}}.
\end{equation*}
Since $\chi_{0}$ is a homomorphism, the map $Y\rightarrow \mathbb{Z}$ defined by $y_{j}\mapsto k_{j}$ is in $X=\mathrm{Hom}(Y,\mathbb{Z})$. Let $x\in X$ be such that for all $j$
\begin{equation*}
    \langle x,y_{j}\rangle=k_{j}.
\end{equation*}

Let $W_{P,\chi}^{\prime}:=\mathrm{Stab}_{W_{P}}(\chi)$.

\subsection{Some Lemmas}

\begin{lemma}
    Fix $u_{0}\in\mathfrak{o}^{\times}$ such that $u_{0}$ represents a primitive root of $\mathfrak{f}^{\times}$. Then the map $Y/Y_{Q,n}\hookrightarrow \overline{T}_{\mathfrak{o}}/\overline{T}_{\mathfrak{o}}\cap Z(\overline{T})$ defined by $y\mapsto \mathbf{s}(y(u_{0}))Z(\overline{T})$ is an injection.
\end{lemma}

\subsection{General case for classical groups}
Not complete. NEed to modify slightly for differn parities of $n$ slightly.
\begin{lemma}
    It suffices to show that
    %
    \begin{align}
        j_{W(C_k)\times W(C_{r-k})}^{W(C_r)}\lambda(2k;n)_C\boxtimes \lambda(2r-2k;n)_C \le \lambda(2r;n)_C. \\
        j_{W(D_k)\times W(B_{r-k})}^{W(B_r)}\lambda(2k;n)_D\boxtimes \lambda(2r-2k+1;n)_B \le \lambda(2r+1;n)_B. \\
        j_{W(D_k)\times W(D_{r-k})}^{W(D_r)}\lambda(2k;n)_D\boxtimes \lambda(2r-2k;n)_D \le \lambda(2r;n)_D. 
    \end{align}
    %
\end{lemma}

\end{comment}

\section{Computation of wave front Set}\label{WFSComp}
In this section we prove the following purely combinatorial result.

\begin{proposition}\label{prop:combo}
    %Fix a pair $(Q,n)$ consisting of a $W$-invariant quadratic form on $Y$ and $n\in \mathbb N$. 
    Let $\nu\in X\otimes \mathbb{R}$ be an exceptional character. % for $(Q,n)$.
    %and assume that $ \Phi $ is irreducible and that the quadratic form $ Q $ takes value 1 on the short coroots. 
    Then
    \begin{equation}\label{CombComp}
        \max_{y\in \mathscr{V}}\mathrm{Sat}_{\mathbb G_y}^{\mathbb G}(\mathcal O_{\rm Spr}(j_{W(\Phi_y^{Q,n})}^{W(\Phi_y)}\sgn)(\mathbb C)) = \mca{O}_{\rm Spr}(j_{W_\nu}^W \sgn)(\mathbb C).
    \end{equation}
\end{proposition}

Our main theorem is an immediate consequence of Propositions \ref{WFThetaRedtoComb} and \ref{prop:combo}.

\begin{theorem}\label{WFThetaThm} 
Suppose that hypotheses \ref{hyp:exp}, \ref{hyp:lift} and \ref{hyp:bilin} hold.
%This is the case if the residue characteristic is larger than some constant depending on the absolute root datum of $\mathbf G$. 

Fix $\pi^{\dagger}$ a distinguished genuine irreducible $\overline{T}$-representation and let $\nu:T\rightarrow \mathbb{C}^{\times}$ be an exceptional character. Let $\Theta$ be the theta representation of $\overline{G}$ attached to $\pi^{\dagger}$ and $\nu$. Then
\begin{equation}\label{WFThetaFinal}
\mathrm{WF}^{s}(\Theta)(\mathbb C)=\mathcal{O}_{\mathrm{Spr}}(j_{W_{\nu}}^{W}(\mathrm{sgn}))(\mathbb C).
\end{equation}
\end{theorem}

We start with a reduction lemma.
\begin{lemma}
    Proposition \ref{prop:combo} holds if it holds whenever $\Phi$ is irreducible and $Q(\alpha^\vee) = 1$ on the short coroots.
\end{lemma}
\begin{proof}
    All of the operations distribute over products so Proposition \ref{prop:combo} holds if it holds whenever $\Phi$ is irreducible.

    Now suppose $\Phi$ is irreducible. 
    Let $\tilde Q$ denote the $W$-invariant quadratic form that takes value $1$ on the short coroots.
    Then since $\Phi$ is irreducible $Q = l \tilde Q$ where $l$ is the value of $Q$ on the short coroots.
    Let $n':=n/\mathrm{GCD}(l,n)$.
    The condition $Q(\alpha^\vee)\alpha(y)\in n\mathbb Z$ is equivalent to $\tilde Q(\alpha^\vee)\alpha(y)\in n'\mathbb Z$ and so
    $$\Phi_y^{Q,n} = \Phi_y^{\tilde Q,n'}.$$
    Moreover 
    $$\alpha^\vee_{Q,n} = \alpha^\vee_{\tilde Q,n'}$$
    and so $\nu$ is also exceptional with respect to $\tilde Q$ and $n'$.
    Using that Proposition \ref{prop:combo} holds for $\tilde Q$ we get that
    $$\max_{y\in \mathscr{V}}\mathrm{Sat}_{\mathbb G_y}^{\mathbb G}(\mathcal O_{\rm Spr}(j_{W(\Phi_y^{Q,n})}^{W(\Phi_y)}\sgn)(\mathbb C)) = \mca{O}_{\rm Spr}(j_{W_{\nu}}^W \sgn)(\mathbb C)$$
    as desired.
\end{proof}

%Notice that $Q(\alpha^{\vee})\alpha(y) \in n\Z$ is equivalent to $\alpha(y) \in n_{\alpha}\Z$. Then if we fix the quadratic form $Q$ such that $Q$ takes value 1 on the short coroots and let $n$ run over $\Z$, it exhausts all cases. We also observe that the computation for each irreducible factor of $\Phi$ is isolated. 
We are thus reduced to prove the case when $\Phi$ is irreducible and $Q(\alpha^\vee)=1$ on the short coroots.
For the rest of this section we will assume that these conditions hold.

\begin{definition}
    Let $ P $ be the coweight lattice of the root system $\Phi$. For $ \alpha \in \Delta$, let $ \omega_{\alpha} $ be the fundamental weight and $ \omega_{\alpha^{\vee}}$ be the fundamental coweight associated with $ \alpha $. 
\end{definition}

We start by showing that the right hand side of line \eqref{CombComp} is attained.

\begin{lemma} \label{lm:Wx_Wy}
	Let $ y=\sum_{\alpha \in \Delta}a_{\alpha} \omega_{\alpha^{\vee}} \in P $.
 If $ x=\sum_{\alpha \in \Delta}\frac{a_{\alpha}Q(\alpha^{\vee})\omega_{\alpha}}{n} $ then $ \Phi_y^{Q,n}=\Phi_x. $
\end{lemma}
\begin{proof}
	Assume $ \beta^{\vee}=\sum_{\alpha \in \Delta}b_{\alpha} \alpha^{\vee} $. Then $ \beta=\sum_{\alpha \in \Phi}\frac{b_{\alpha}Q(\alpha^{\vee})\alpha}{Q(\beta^{\vee})} $, since the length ratio between $ \alpha^{\vee} $ and $ \beta^{\vee} $ is equal to $ Q(\alpha^{\vee})/Q(\beta^{\vee}) $.
	
	We have $$ \langle y,\beta/n_{\beta} \rangle =\sum_{\alpha \in \Delta}\frac{a_{\alpha}b_{\alpha}Q(\alpha^{\vee})}{Q(\beta^{\vee})}\cdot\frac{\text{gcd}(n,Q(\beta^{\vee}))}{n} $$and
	$$ \langle \beta^{\vee},x \rangle=\sum_{\alpha \in \Delta}\frac{a_{\alpha}b_{\alpha}Q(\alpha^{\vee})}{Q(\beta^{\vee})}\cdot\frac{Q(\beta^{\vee})}{n}.  $$
	
	Since $ \frac{b_{\alpha}Q(\alpha^{\vee})}{Q(\beta^{\vee})}$ is an integer for each $ \alpha $, it is clear that $  \langle y,\beta/n_{\beta} \rangle \in \Z$ if and only if $ \langle \beta^{\vee},x \rangle \in \Z $. Thus $ \Phi_y^{Q,n}=\Phi_x. $ 
\end{proof}
\begin{corollary}
    There exists a $y\in \mathscr V$ such that 
    $$W(\Phi_y) = W, \quad W(\Phi_y^{Q,n}) = W_\nu.$$
\end{corollary}
\begin{proof}
    Let $ a_{\alpha} \in \Z $ be such that $ \frac{a_{\alpha}Q(\alpha^{\vee}) }{\text{gcd}(n,Q(\alpha^{\vee}))}\equiv 1 \text{ mod } n_{\alpha}$ and let $ y=\sum_{\alpha \in \Delta}a_{\alpha} \omega_{\alpha^{\vee}} \in P $.
    Since $y\in P$, we have that $W(\Phi_y) = W$.
    By Lemma \ref{lm:Wx_Wy}, we see that $ \Phi_y^{Q,n}=\Phi_{\nu} $.
\end{proof}

Now we prove that the left hand side of \eqref{CombComp} is bounded above by the right hand side.

\begin{lemma}\label{lem:max-reduction}
    $$\max_{y\in \mathscr{V}}\mathrm{Sat}_{\mathbb G_y}^{\mathbb G}(\mathcal O_{\rm Spr}(j_{W(\Phi_y^{Q,n})}^{W(\Phi_y)}\sgn)(\mathbb C)) = \max_{y\in (\mathscr{V}\cap \overline{\mathscr A}_0)+P}\mathrm{Sat}_{\mathbb G_y}^{\mathbb G}(\mathcal O_{\rm Spr}(j_{W(\Phi_y^{Q,n})}^{W(\Phi_y)}\sgn)(\mathbb C)).$$
\end{lemma}
\begin{proof}
    Let $ y \in \mathscr{V} $ and $ y^{\prime}=w(y+y_0) $ with $ w \in W,y_0 \in Y^{sc}_{Q,n}$. We observe that \begin{equation} \label{eq:wy_equality}
    	\mathrm{Sat}_{\mathbb G_{ y^{\prime}}}^{\mathbb G}(\mathcal O_{\rm Spr}(j_{W(\Phi_{y^{\prime}}^{Q,n})}^{W(\Phi_{ y^{\prime}})}\sgn)(\mathbb C))=\mathrm{Sat}_{\mathbb 
     G_y}^{\mathbb  G}(\mathcal O_{\rm Spr}(j_{W(\Phi_y^{Q,n})}^{W(\Phi_y)}\sgn)(\mathbb C)).
    \end{equation}
    Since for any $y\in \mathscr V$ there exists $ w \in W $ and $ p \in P $ such that $ wy+p \in \overline{\mathscr A}_0$, by \eqref{eq:wy_equality} we can just consider the vertices in $ (\mathscr{V} \cap \overline{\mathscr A}_0)+P$.    
\end{proof}

We now show that $\Phi_y$ and $\Phi_y^{Q,n}$ takes a very particular form when $y\in (\mathscr V\cap \overline{\mathscr A}_0)+P$.

\begin{lemma} \label{lem:weyl-lemma}
    Let $v\in \mathscr V\cap \overline{\mathscr A}_0$ and $ y=v+p $ for $p\in P$.
    \begin{enumerate}
        \item Let $ \widetilde{\alpha} $ be the highest root of $\Phi$ with respect to $ \Delta $. 
        The root system $ \Phi_v $ has a set of simple roots $ \Delta_v$ lying in the set $\Delta\cup\{-\tilde\alpha\}$. 
        It consists of those $\alpha\in \Delta$ that vanish on $v$ and includes $-\tilde\alpha$ if $\tilde\alpha(v) = 1$.
        \item Let 
        $$ \Phi^{\prime}_v=\{\alpha_{Q,n}:\alpha \in \Phi_v\},\quad \Delta^{\prime}_v=\{\alpha_{Q,n}:\alpha \in \Delta_v\},$$and $ \widetilde{\Delta}^{\prime}_v $ be the set of highest roots in $ \Phi^{\prime}_v $ with respect to $ \Delta^{\prime}_v $. Define 
        $$ \overline{\mathscr A}^{Q,n}_{v,0}:=\{y \in Y \otimes \R: \alpha(y)\ge 0 \text{ for } \alpha \in \Delta_{v} \text{ and } \tilde\beta(y)\le 1  \text{ for } \tilde\beta \in \widetilde{\Delta}^{\prime}_v \}.$$
        Then there exists $ w \in W(\Phi_v) $ and $ y_0 \in Y^{sc}_{Q,n} $ such that $ w(y+y_0) \in  \overline{\mathscr A}^{Q,n}_{v,0}$. 
        \item\label{weyl-lemma3} Let 
        $$ \overline{\Phi}^{Q,n}_{y}:=\{\alpha_{Q,n}:\alpha \in \Phi^{Q,n}_y\}.$$
        If $ y \in (v+P) \cap \overline{\mathscr A}^{Q,n}_{v,0}$, then $ \overline{\Phi}^{Q,n}_y $ has a set of simple roots 
        $$\overline{\Delta}^{Q,n}_y \subset \Delta^{\prime}_v \cup \{-\tilde\beta:\tilde\beta \in \widetilde{\Delta}^{\prime}_v\}$$
        containing $\beta\in \Delta_v'$ if $\beta(y) = 0$ and $-\tilde\beta$ for $\tilde\beta\in \tilde\Delta_v'$ if $\tilde\beta(y)=1$.
        Thus we can describe $\overline{\Delta}^{Q,n}_y$ using a subdiagram of the extended Dynkin diagram associated to $ \Delta^{\prime}_v $.
    \end{enumerate}
\end{lemma}
\begin{proof}
    These are straighforward consequences of basic properties of affine Weyl groups.
\end{proof}

\begin{definition}$ $
    \begin{enumerate}
        \item Let $ \overline{\mathscr A}^{Q,n}_0 $ denote $ \overline{\mathscr A}^{Q,n}_{0,0} $ and let $ \widetilde{\alpha}_{Q,n} $ denote the unique element in $ \widetilde{\Delta}^{\prime}_0 $.
        \item For vertices $v,v'\in \mathscr V\cap \mathscr A_0$ we say $v\sim v'$ if there exists $w \in W$ such that $wv-v'\in P$.
    \end{enumerate}
\end{definition}

\begin{corollary}\label{cor:EquivRed}
    %Then for $ y=v+p $, there exists $ w \in W(\Phi_v) $ and $ y_0 \in Y^{sc}_{Q,n} $ such that $ w(y+y_0) \in  \mathscr A^{Q,n}_{v,0}$. 
    %
    $$\max_{y\in \mathscr{V}}\mathrm{Sat}_{\mathbb G_y}^{\mathbb G}(\mathcal O_{\rm Spr}(j_{W(\Phi_y^{Q,n})}^{W(\Phi_y)}\sgn)(\mathbb C)) = \max_{\substack{v\in \mathscr V\cap \overline{\mathscr A}_0/\sim \\ y\in (v+P)\cap \overline{\mathscr A}_{v,0}^{Q,n}}}\mathrm{Sat}_{\mathbb G_y}^{\mathbb G}(\mathcal O_{\rm Spr}(j_{W(\Phi_y^{Q,n})}^{W(\Phi_y)}\sgn)(\mathbb C)).$$
    %
    %Again by $ \eqref{eq:wy_equality} $, we see it is enough to compute the item for elements in each $ (v+P) \cap A^{Q,n}_{v,0}$ for $ v \in \mathscr{V}\cap \mathscr A_0 $. 
    Moreover, all hyperspecial vertices are $\sim$-equivalent and so we may take $v=0$ for the $\sim$-equivalence class of hyperspecial vertices.
    
    %Let $ \overline{\Phi}^{Q,n}_{y}:=\{\alpha_{Q,n}:\alpha \in \Phi^{Q,n}_y\}$. Note that if $ y \in (v+P) \cap \mathscr A^{Q,n}_{v,0}$, then $ \overline{\Phi}^{Q,n}_y $ has a set of simple roots $\overline{\Delta}^{Q,n}_y \subset \Delta^{\prime}_v \cup \{-\beta:\beta \in \widetilde{\Delta}^{\prime}_v\} $. Thus we can describe it using a subdiagram of the extended Dynkin diagram associated with $ \Delta^{\prime}_v $.
\end{corollary}
\begin{proof}
    By Lemma \ref{lem:max-reduction} we may restrict the maximum to range over $y\in (\mathscr V\cap \overline{\mathscr A}_0)+P$. Note that by equation \eqref{eq:wy_equality} $v$ can be chosen freely up to $\sim$-equivalence.
    Let $y=v+p$ where $v\in \mathscr V\cap \overline{\mathscr A}_0$ and $p\in P$.
    By Lemma \ref{lem:weyl-lemma}, \eqref{weyl-lemma3} we may conjugate $y$ by $W(\Phi_v)\ltimes Y^{sc}_{Q,n}$ to lie in $(v+P)\cap \overline{\mathscr A}_{v,0}^{Q,n}$ and again by equation \eqref{eq:wy_equality} this does not change the quantity being maximised.

    The last statement follows from the fact for split groups that $\Phi_v =\Phi$ if and only if $v$ hyperspecial.
\end{proof}
Finally we apply Corollary \ref{cor:EquivRed} to each Cartan type to complete Theorem \ref{WFThetaThm}.

\subsection{Type $ A_r $}
Label the simple roots as in the following Dynkin diagram:
\begin{center}
	\dynkin[labels={1,2,3,r-1,r},
	scale=2] A{ooo...oo}
\end{center}
In this case all vertices are hyperspeical so we may assume $v=0$.
Let $y\in 0+P = P$. 
It is of the form $  \sum^r_{i=1}b_{i} \omega_{\alpha_i^{\vee}}$ with $ b_{i}\in \Z $. 

Since $ \widetilde{\alpha}_{Q,n}=\sum^r_{i=1}\alpha_i/n $, we see $ y $ lies in $ \overline{\mathscr A}^{Q,n}_0 $ if and only if each $ b_i \ge 0$ and $ \sum^r_{i=1}b_{i} \le n$. 
But $ \alpha_i/n \notin \overline{\Delta}^{Q,n}_y $ if $ b_i>0 $ and $ -\widetilde{\alpha}_{Q,n} \notin \overline{\Delta}^{Q,n}_y$ if $ \sum^r_{i=1} b_i <n $. 
From this we see readily that 
$$y\in \overline{\mathscr A}_0^{Q,n} \implies \#(\Delta_0'\cup\{-\tilde\alpha_{Q,n}\})-\#(\overline\Delta_{y}^{Q,n})\le n.$$
In other words we can discard at most $ n $ elements from $\Delta^{\prime}_0 \cup \{-\widetilde{\alpha}_{Q,n}\}$ to obtain $\overline{\Delta}^{Q,n}_y$.

Thus $ \Phi^{Q,n}_y $ is of the type $ A_{r_1-1}\times  A_{r_2-1}\times... \times  A_{r_m-1} $ with $\sum_{i=1}^{m}r_i=r+1 $ and $ m\le n $. 

Recall that every nilpotent orbit for type $ A_r $ is parametrized by a partition $a=(a_1,a_2,...,a_k) $ of $ r+1 $ such that $ a_1 \ge a_2 \ge ...\ge a_k $. We define $ L(a)=a_1 $. 

Let $ d=(d_1,d_2,...,d_k) $ be the dual partition of $ (r_1,r_2,...,r_m) $. 
By \cite[\S 13.3]{Carter} we have that 
$$ \mca{O}_{\rm Spr}(j_{W(\Phi_y^{Q,n})}^{W(\Phi_y)}\sgn)=d.$$ 
Since $ m\le n $, it is clear that $ L(d) \le n$.

Let $ \lambda(m_1;m_2) $ be the largest partition of $ m_1 $ satisfying $ L(-) \le m_2 $. Then we have $d \le \lambda(r+1;n) $. But it is shown in \cite[Theorem 4.3]{GT22} that $ \mca{O}_{\rm Spr}(j_{W_\nu}^W \text{sgn})=\lambda(r+1;n) $. Therefore, the main proposition holds for type $ A_r $.

\subsection{Type $ D_r $}
Label the simple roots as in the following Dynkin diagram:
\begin{center}
	\dynkin[labels={1,2,3,,r-2,r-1,r},
	label directions={,,,,right,,},
	scale=1.8] D{ooo...oooo}
\end{center}

\subsubsection{Hyperspecial} \label{h_D}
First we consider the hyperspecial vertices. 
We take $v=0$ and fix $ y =\sum^r_{i=1}b_{i} \omega_{\alpha_i^{\vee}} $ in the coweight lattice $ P $. 
Since $\widetilde{\alpha}_{Q,n}=\alpha_1/n+2\alpha_2/n+...+2\alpha_{r-2}/n+\alpha_{r-1}/n+\alpha_{r}/n$, we see $ y \in \overline{\mathscr A}^{Q,n}_0 $ if and only if each $ b_i\ge 0 $ and $ b_1+2b_2+...+2b_{r-2}+b_{r-1}+b_{r} \le n$.

After removing elements in $\Delta^{\prime}_0 \cup \{-\widetilde{\alpha}_{Q,n}\}$, we observe that $ \overline{\Phi}^{Q,n}_y $ is of the type $ D_{r_1} \times A_{r_2-1} \times A_{r_3-1} \times ... \times A_{r_{k-1}-1}\times D_{r_k}$ with $ \sum_{i=1}^{k}r_i=r $. Clearly $ \Phi^{Q,n}_y $ is of the same type.

Let $$ a=(r_1,\lfloor (r_2+1)/2 \rfloor,\lfloor (r_3+1)/2 \rfloor,...,\lfloor (r_{k-1}+1)/2 \rfloor,r_k) \text{ and } b=(\lfloor r_2/2 \rfloor,\lfloor r_3/2 \rfloor,...,\lfloor r_{k-1}/2 \rfloor) $$be two partitions. Let $ a^{*} $ and $ b^{*} $ be the dual partitions of $ a,b $ respectively. Then \cite{Mac1} and \cite[\S 11.4]{Carter} imply that  $ j_{W(\Phi_y^{Q,n})}^{W(\Phi_y)}\sgn $ correspondes to the an unordered pair $ (a^{*},b^{*}) $.

Each nilpotent orbit in type $ D_r $ is associated with a partition $ \lambda $ of $ 2r $ in which even part occurs an even number of times. There is a procedure to construct an unordered pair $ \phi_D(\lambda)=(\xi,\eta) $ with $ |\xi|+|\eta|=r $ from $ \lambda $, for details see \cite[\S 13.3]{Carter}.

Let $ \lambda=(\lambda_1,\lambda_2,...) $ with $ \lambda_1 \ge \lambda_2\ge...$ be a partition of type $ D $. One can easily verify that its corresponding unordered pair $ \phi_D(\lambda)=(\xi,\eta) $ satisfies the following properties: 
\begin{itemize}
	\item[(i)] If $ \lambda_1 $ is odd, then either $ L(\xi) $ or $ L(\eta) $ is equal to $ (\lambda_1+1)/2 $. 
	
	\item[(ii)] If $ \lambda_1 $ is odd and $ \lambda_2=\lambda_1 $, then $$ L(\xi)=(\lambda_1+1)/2,L(\eta)=(\lambda_1-1)/2 \text{ or } L(\eta)=(\lambda_1+1)/2,L(\xi)=(\lambda_1-1)/2 .$$
	
	\item[(iii)] If $ \lambda_1 $ is even, which leads to $ \lambda_2=\lambda_1 $, then $ L(\xi)=L(\eta)=\lambda_1/2 $.
\end{itemize}

If $ \phi_D(\lambda)=(a^{*},b^*) $, then $\mca{O}_{\rm Spr}(j_{W(\Phi_y^{Q,n})}^{W(\Phi_y)}\sgn)$ is parametrized by $ \lambda $. Notice  $ \alpha_i/n \notin \overline{\Delta}^{Q,n}_y $ if $ b_i>0 $, and $ -\widetilde{\alpha}_{Q,n} \notin \overline{\Delta}^{Q,n}_y$ if $b_1+2b_2+...+2b_{r-2}+b_{r-1}+b_{r} <n $. We describe the type of $ \overline{\Phi}^{Q,n}_y $ by discarding elements in $\Delta^{\prime}_0 \cup \{-\widetilde{\alpha}_{Q,n}\}$ and list the following cases:

\begin{itemize}
	\item[(I)] When $ n $ is odd, if we discard $ t $ elements in $ \{-\widetilde{\alpha}_{Q,n},\alpha_1/n,\alpha_{r-1}/n,\alpha_{r}/n\} $, then we can discard at most $ \lfloor (n-t)/2 \rfloor$ elements in $ \{\alpha_2/n,\alpha_3/n,...,\alpha_{r-2}/n\} $. It is direct to see $ L(a^*) \le (n+1)/2 $ and $ L(b^*) \le (n-1)/2 $ by the procedure of giving the corresponding unordered pair $ (a^*,b^*) $.
	
	\item[(II)] When $ n $ is even, if we discard one element in $  \{-\widetilde{\alpha}_{Q,n},\alpha_1/n\} $, one element in $ \{\alpha_{r-1}/n,\alpha_{r}/n\}$ and at most $ n/2-1 $ elements in $ \{\alpha_2/n,\alpha_3/n,...,\alpha_{r-2}/n\} $. Then $ L(a^*) \le n/2 $ and $ L(b^*) \le n/2 $. In other cases, $ L(a^*) \le n/2+1 $ and $ L(b^*) \le n/2-1 $. In particular, we note that discarding 2$ \varepsilon_1 $ elements in $\{-\widetilde{\alpha}_{Q,n},\alpha_1/n\} $, 2$ \varepsilon_2 $ elements in $ \{\alpha_{r-1}/n,\alpha_{r}/n\}$ and $ n/2-\varepsilon_1-\varepsilon_2 $ elements in $ \{\alpha_2/n,\alpha_3/n,...,\alpha_{r-2}/n\} $ is a necessary and sufficient condition for $ L(a^*)=n/2+1 $, where $ \varepsilon_1,\varepsilon_2 \in \{0,1\} $.
\end{itemize}

For a partition $ \lambda $, let $ \lambda_D $ be its $ D $-collapse. By the three properties demonstrated above and the discussions (I)(II), we get the following results:

When $ n $ is odd, $\mca{O}_{\rm Spr}(j_{W(\Phi_y^{Q,n})}^{W(\Phi_y)}\sgn) \le \lambda(2r;n)_D$.

When $ n $ is even, $\mca{O}_{\rm Spr}(j_{W(\Phi_y^{Q,n})}^{W(\Phi_y)}\sgn) \le (n+1,\lambda(2r-n-1;n))_D$.

It is shown in \cite[\S 4.4]{glt} that $ \mca{O}_{\rm Spr}(j_{W_\nu}^W \text{sgn})$ reaches the upper bound in each case.

\subsubsection{Non-hyperspecial} \label{nh_D}
Let $ 2\le s \le r-2$. Consider $ v \in \overline{\mathscr A}_0 $ such that $ \widetilde{\alpha}(v)=1 $ and $ \alpha_i(v)=0 $ for $ i \neq s $. Then $ v $ is a non-hyperspecial vertex and $ \Phi_v $ is of type $ D_{s}\times D_{r-s} $ with a set of simple roots $\{-\widetilde{\alpha},\alpha_1,...,\alpha_{s-1}\}\cup \{\alpha_{s+1},...,\alpha_{r-1},\alpha_{r}\}$.

Let $y=\sum^r_{i=1}b_{i} \omega_{\alpha_i^{\vee}} \in P$ and $ v+y \in \overline{\mathscr A}^{Q,n}_{v,0} $. First we compute $\mca{O}_{\rm Spr}(j_{W(\Phi_{v+y}^{Q,n})}^{W(\Phi_{v+y})}\sgn)$.

When $ n $ is odd, clearly it is not greater than $ \lambda(2s;n)_D \times \lambda(2(r-s);n)_D $ from the result in hyperspecial case.

When $ n $ is even, assume $\mca{O}_{\rm Spr}(j_{W(\Phi_{v+y}^{Q,n})}^{W(\Phi_{v+y})}\sgn)=\lambda_1 \times \lambda_2$. For $ v+y \in \mathscr A^{Q,n}_{v,0}$, we see $ L(\lambda_1)=L(\lambda_2)=n+1$ implies that $ (-\widetilde{\alpha}+\alpha_1)(v+y)=0 \text{ or } 2$ and $ (\alpha_{r-1}+\alpha_r)(v+y)=0 \text{ or } 2$ by the argument in (II). But since we observe that$$ (-\widetilde{\alpha}+\alpha_1+\alpha_{r-1}+\alpha_r)(v+y) \equiv 1 \text{ mod } 2, $$ then if either of $ L(\lambda_1) $ and $ L(\lambda_2) $ is equal to $ n+1 $, the other one must be less than $ n+1 $.

Now we recall the procedure to attach a weighted Dynkin diagram to a nilpotent orbit of type $ D_r $ in \cite[\S 13.1]{Carter}. Let $ \lambda=(\lambda_1,\lambda_2,...) $ be a partition of $ 2r $ in type $ D $. For each $ \lambda_i $ in the partition, we take the set of integers $ \lambda_i-1,\lambda_i-3,...,3-\lambda_i,1-\lambda_i $. Then we take the union of these sets and write it as $ (\xi_1,\xi_2,..,\xi_{2r}) $ with $ \xi_1 \ge \xi_2 \ge ... \ge \xi_{2r} $. 

Notice that $ \xi_1 \ge \xi_2 \ge ... \ge \xi_r \ge 0 $. We attatch the weighted Dynkin diagram to $ \lambda$ given by
\begin{center}
	\dynkin[labels={ \xi_1-\xi_2 ,\xi_2-\xi_3,, \xi_{r-2}-\xi_{r-1},\xi_{r-1}-\xi_{r},\xi_{r-1}+\xi_{r}},edge length=.55cm,
	label directions={,,,right,,},
	scale=1.8] D{ooo...ooo}
\end{center}
or 
\begin{center}
	\dynkin[labels={ \xi_1-\xi_2 ,\xi_2-\xi_3,, \xi_{r-2}-\xi_{r-1},\xi_{r-1}+\xi_{r},\xi_{r-1}-\xi_{r}},edge length=.55cm,
	label directions={,,,right,,},
	scale=1.8] D{ooo...ooo}
\end{center}
If there is an odd element in $ \lambda $, then $ \xi_r=0 $ and the two diagrams are the same. We observe that 	$L(\lambda)=\xi_1+1=-\xi_{2r}+1 $.

Suppose $\mca{O}_{\rm Spr}(j_{W(\Phi_{v+y}^{Q,n})}^{W(\Phi_{v+y})}\sgn)=\lambda_1 \times \lambda_2$ and its saturation is $ \lambda $. Since the action of elements in $ W $ can only send $ \xi_i $ to an element of the form $ \pm \xi_{\sigma(i)} $, we see the integers for $ \lambda $ actually come from the set of integers for $ \lambda_1 $ and $ \lambda_2 $.

As discussed above, we have $ \lambda_1 \le \lambda(2s;n)_D $ and $\lambda_2 \le \lambda(2(r-s);n)_D $ when $ n $ is odd. Thus the integers for $ \lambda_1 $ and $ \lambda_2 $ are not greater than $ n-1 $. Then $ \lambda \le \lambda(2r;n)_D$.

When $ n $ is even, there is at most one element equal to $ n $ in the set of integers for $ \lambda_1 $ and $ \lambda_2 $. Thus we have $ \lambda \le (n+1,\lambda(2r-n-1;n))_D$.

\subsection{Type $ B_r $}
Label the simple roots as in the following Dynkin diagram:
\begin{center}
	\dynkin[labels={1,2,3,,r-1,r},
	scale=1.8] B{ooo...ooo}
\end{center}
\subsubsection{Hyperspecial} 
Fix $ y =\sum^r_{i=1}c_{i} \omega_{\alpha_i^{\vee}} \in P \cap \overline{\mathscr A}^{Q,n}_0 $.

When $ n $ is odd, $ \Phi_{Q,n} $ is of type $ B_r $ and $ \widetilde{\alpha}_{Q,n}=\alpha_1/n+2\alpha_2/n+2\alpha_3/n+...+2\alpha_r/n$.

When $ n $ is even, $ \Phi_{Q,n} $ is of type $ C_r $ and $ \widetilde{\alpha}_{Q,n}=2\alpha_1/n+2\alpha_2/n+2\alpha_3/n+...+2\alpha_r/n$.

Let $ a=(a_1,a_2,...,a_k) $ and $ b=(b_1,b_2,...,b_l) $ be two partitions such that $ |a|+|b|=r $. Then $ j^{B_r}_{D_{a_1}\times ... \times  D_{a_k} \times B_{b_1}\times ... \times  B_{b_l}}\text{sgn}$ correspondes to an ordered pair $ (a^*,b^*) $, see \cite[\S 11.4]{Carter}. 

By \cite{Mac1}, one has \begin{equation*} \label{induction:B}
	j^{B_k}_{A_{k-1}}\text{sgn}=j^{B_k}_{D_{\lfloor (k+1)/2 \rfloor}\times B_{\lfloor k/2 \rfloor}} \text{sgn}.
\end{equation*}

Assume that $ (a^*,b^*) $ is the partition associated with $ j_{W(\Phi_y^{Q,n})}^{W(\Phi_y)}\sgn $. Then similar to type $D$, a brief discussion leads to the following results:

\begin{itemize}
	\item[(I)] When $ n $ is odd, $ L(a^*)\le (n-1)/2 $, $ L(b^*) \le (n+1)/2 $.
	
	\item[(II)] When $ n $ is even, $ L(a^*)\le n/2-1 $, $ L(b^*) \le n/2+1 $.
\end{itemize}

Every nilpotent orbit in type $ B_r $ is parametrized by a partition $ \lambda $ of $ 2r+1 $ in which even part occurs an even number of times. There is a map from one partition $ \lambda $ of type $ B_r $ to an ordered pair $ \phi_B(\lambda)=(\xi,\eta) $ with $ |\xi|+|\eta|=r$, see \cite[\S 13.3]{Carter}. The map satisfies the following properties:

\begin{itemize}
	\item[(i)] If $ L(\lambda) $ is odd, then $ L(\xi)=(L(\lambda)-1)/2 $.
	\item[(ii)] If $ L(\lambda) $ is even, then $ L(\eta)=L(\lambda)/2+1$.
\end{itemize}

If $ \phi_B(\lambda)=(a^{*},b^*) $, then $\mca{O}_{\rm Spr}(j_{W(\Phi_y^{Q,n})}^{W(\Phi_y)}\sgn)$ is parametrized by $ \lambda $. By (I)(II) and (i)(ii) listed above, one has:

$$\mca{O}_{\rm Spr}(j_{W(\Phi_y^{Q,n})}^{W(\Phi_y)}\sgn) \le \lambda(2r+1;n)_B,$$where $ \lambda(2r+1;n)_B $ is the $ B $-collapse of $ \lambda(2r+1;n) $.

It is shown in \cite[P38]{glt} that $ \mca{O}_{\rm Spr}(j_{W_\nu}^W \text{sgn})$ reaches the upper bound.

\subsubsection{Non-hyperspecial}
Let $ 2\le s \le r$. Consider $ v \in \overline{\mathscr A}_0 $ such that $ \widetilde{\alpha}(v)=1 $ and $ \alpha_i(v)=0 $ for $ i \neq s $. Then $ v $ is a non-hyperspecial vertex and $ \Phi_v $ is of type $ D_{s}\times B_{r-s} $ with a set of simple roots $\{-\widetilde{\alpha},\alpha_1,...,\alpha_{s-1}\}\cup \{\alpha_{s+1},...,\alpha_{r-1},\alpha_{r}\}$. Take $ y=\sum^r_{i=1}c_{i} \omega_{\alpha_i^{\vee}} \in P$ and assume $ v+y \in \overline{\mathscr A}^{Q,n}_{v,0} $.  

When $ n $ is odd, one has $\mca{O}_{\rm Spr}(j_{W(\Phi_{v+y}^{Q,n})}^{W(\Phi_{v+y})}\sgn) \le  \lambda(2s;n)_D \times \lambda(2(r-s)+1;n)_B $ by the results in the hyperspecial cases of type $ B_{r-s} $ and type $ D_s $.

When $ n $ is even, assume $\mca{O}_{\rm Spr}(j_{W(\Phi_{v+y}^{Q,n})}^{W(\Phi_{v+y})}\sgn)=\lambda_1 \times \lambda_2$. By the result in the hyperspecial case, one has $ \lambda_2 \le \lambda(2(r-s)+1;n)_B $. Since $(-\widetilde{\alpha}+\alpha_1)(v+y) \equiv 1 \text{ mod } 2,$ we have $ L(\lambda_1) <n+1 $ by (II) in \S \ref{h_D}. 

The procedure to attach a weighted Dynkin diagram to a nilpotent orbit of type $ B$ is almost the same as type $ D $. Let $ \lambda=(\lambda_1,\lambda_2,...) $ be a partition of $ 2r+1 $ in type $ B $. For each $ \lambda_i $ in the partition, we take the set of integers $ \lambda_i-1,\lambda_i-3,...,3-\lambda_i,1-\lambda_i $. Then we take the union of these sets and write it as $ (\xi_1,\xi_2,..,\xi_{2r+1}) $ with $ \xi_1 \ge \xi_2 \ge ... \ge \xi_{2r+1} $. 

We attatch the weighted Dynkin diagram to $ \lambda$ given by
\begin{center}
	\dynkin[labels={ \xi_1-\xi_2 ,\xi_2-\xi_3,, ,\xi_{r-1}-\xi_{r},\xi_{r}},edge length=.55cm,
	label directions={,,,right,,},
	scale=1.8] B{ooo...ooo}
\end{center}

Suppose $ \lambda $ is the saturation of $ \lambda_1 \times \lambda_2 $. The action of Weyl elements can only send $ \xi_i $ to an element of the form $ \pm \xi_{\sigma(i)} $. Thus as the case of type $ D $, the integers for $ \lambda $ actually come from the set of integers for $ \lambda_1 $ and $ \lambda_2 $. Then one has $ \lambda \le \lambda(2r+1;n) $

\subsection{Type $ C_r $}
Label the simple roots as in the following Dynkin diagram:
\begin{center}
	\dynkin[labels={1,2,3,,r-1,r},
	scale=1.8] C{ooo...ooo}
\end{center}
\subsubsection{Hyperspecial} 
Fix $ y =\sum^r_{i=1}c_{i} \omega_{\alpha_i^{\vee}} \in P \cap \overline{\mathscr A}^{Q,n}_0$.

When $ n $ is odd, $ \Phi_{Q,n} $ is of type $ C_r $ and $ \widetilde{\alpha}_{Q,n}=2\alpha_1/n+2\alpha_2/n+...+2\alpha_{r-1}/n+\alpha_r/n$.

When $ n $ is even, $ \Phi_{Q,n} $ is of type $ B_r $ and $ \widetilde{\alpha}_{Q,n}=2\alpha_1/n+4\alpha_2/n+4\alpha_3/n+...+4\alpha_{r-1}/n+2\alpha_r/n$.

Let $ a=(a_1,a_2,...,a_k) $ and $ b=(b_1,b_2,...,b_l) $ be two partitions such that $ |a|+|b|=r $. Similar to type $ B_r $, the representation $ j^{C_r}_{D_{a_1}\times ... \times  D_{a_k} \times C_{b_1}\times ... \times  C_{b_l}}\text{sgn}$ correspondes to the ordered pair $ (a^*,b^*) $ and one has 
$$j^{C_k}_{A_{k-1}}\text{sgn}=j^{C_k}_{D_{\lfloor (k+1)/2 \rfloor}\times C_{\lfloor k/2 \rfloor}} \text{sgn}.$$

Assume that $ (a^*,b^*) $ is the partition associated with $ j_{W(\Phi_y^{Q,n})}^{W(\Phi_y)}\sgn $. We list the following cases:

\begin{itemize}
	\item[(I)] When $ n $ is odd, $ L(a^*)\le (n-1)/2 $, $ L(b^*) \le (n+1)/2 $.
	
	\item[(II)] When $ n$ is even and $ n/2 $ is even, $ L(a^*)\le n/4 $, $ L(b^*) \le n/4 $.
	
	\item[(III)] When $ n$ is even and $ n/2 $ is odd, if we discard $ \alpha_r/n $, then $ L(a^*)\le (n+2)/4,L(b^*) \le (n-2)/4$. Otherwise, we have $ L(a^*)\le (n-2)/4,L(b^*) \le (n+2)/4$.
\end{itemize}

Every nilpotent orbit in type $ C_r $ is parametrized by a partition $ \lambda $ of $ 2r$ in which even odd occurs an even number of times. As before, a partition $ \lambda=(\lambda_1,\lambda_2,...,\lambda_t) $ with $ \lambda_1 \ge \lambda_2 \ge ... \ge \lambda_t$ of type $ C_r $ to an ordered pair $ \phi_C(\lambda)=(\xi,\eta) $ with $ |\xi|+|\eta|=r$. It satisfies the following properties:

\begin{itemize}
	\item[(i)] If $ \lambda_1 $ is odd, then $ L(\eta) $ is equal to $ (\lambda_1+1)/2 $. 
	
	\item[(ii)] If $ \lambda_1 $ is even, then $ L(\xi) $ is equal to $ \lambda_1/2 $. 
	
	\item[(iii)] If $ \lambda_1 $ is even and $ \lambda_2=\lambda_1 $, then $ L(\xi)=L(\eta)=\lambda_1/2 $.
\end{itemize}

If $ \phi_C(\lambda)=(a^{*},b^*) $, then $\mca{O}_{\rm Spr}(j_{W(\Phi_y^{Q,n})}^{W(\Phi_y)}\sgn)$ is parametrized by $ \lambda $. By (I)(II)(III) and (i)(ii)(iii) listed above, one has:

When $ n $ is odd, $\mca{O}_{\rm Spr}(j_{W(\Phi_y^{Q,n})}^{W(\Phi_y)}\sgn) \le \lambda(2r;n)_C$.

When $ n $ is even and $ n/2 $ is even, $\mca{O}_{\rm Spr}(j_{W(\Phi_y^{Q,n})}^{W(\Phi_y)}\sgn) \le \lambda(2r;n/2)_C$.

When $ n $ is even and $ n/2 $ is odd, $\mca{O}_{\rm Spr}(j_{W(\Phi_y^{Q,n})}^{W(\Phi_y)}\sgn) \le (n/2+1,\lambda(2r-n/2-1;n/2))_C$. 

We note that in the third case, $ L(\mca{O}_{\rm Spr}(j_{W(\Phi_y^{Q,n})}^{W(\Phi_y)}\sgn))=n/2+1 $ requires $ L(a^*) $ to be $ (n+2)/4 $, which implies $ c_r=1 $.

It is shown in \cite[\S 4]{GT22} and \cite[\S 4.4]{glt} that $ \mca{O}_{\rm Spr}(j_{W_\nu}^W \text{sgn})$ reaches the upper bound in each case.

\subsubsection{Non-hyperspecial}
Let $ 2\le s \le r$. Consider $ v \in \overline{\mathscr A}_0 $ such that $ \widetilde{\alpha}(v)=1 $ and $ \alpha_i(v)=0 $ for $ i \neq s $. Then $ v $ is a non-hyperspecial vertex and $ \Phi_v $ is of type $ C_{s}\times C_{r-s} $ with a set of simple roots $\{-\widetilde{\alpha},\alpha_1,...,\alpha_{s-1}\}\cup \{\alpha_{s+1},...,\alpha_{r-1},\alpha_{r}\}$. Take $ y=\sum^r_{i=1}c_{i} \omega_{\alpha_i^{\vee}} \in P$ and assume $ v+y \in \overline{\mathscr A}^{Q,n}_{v,0} $. 

When $ n $ is odd, one has $\mca{O}_{\rm Spr}(j_{W(\Phi_{v+y}^{Q,n})}^{W(\Phi_{v+y})}\sgn) \le  \lambda(2s;n)_C \times \lambda(2(r-s);n)_C $ by the results in the hyperspecial case. 

When $ n $ is even and $ n/2 $ is even, it is also clear that $\mca{O}_{\rm Spr}(j_{W(\Phi_{v+y}^{Q,n})}^{W(\Phi_{v+y})}\sgn) \le  \lambda(2s;n/2)_C \times \lambda(2(r-s);n/2)_C $.

When $ n $ is even and $ n/2 $ is odd, assume $\mca{O}_{\rm Spr}(j_{W(\Phi_{v+y}^{Q,n})}^{W(\Phi_{v+y})}\sgn)=\lambda_1 \times \lambda_2$. We see $ L(\lambda_1)=L(\lambda_2)=n/2+1$ implies that $ (-\widetilde{\alpha}+\alpha_r)(v+y)=2$. But $(-\widetilde{\alpha}+\alpha_r)(v+y) \equiv 1 \text{ mod } 2.$ Then if either of $ L(\lambda_1) $ and $ L(\lambda_2) $ is equal to $ n/2+1 $, the other one must be less than $ n/2+1 $.

As before, by \cite[\S 13.1]{Carter} one can write down the weighted Dynkin diagram for the saturation of $ \mca{O}_{\rm Spr}(j_{W(\Phi_{v+y}^{Q,n})}^{W(\Phi_{v+y})}\sgn) $. Using the same argument as type $ B $ and type $ D $, we see the upper bound listed in the hyperspecial also holds for the saturation. 

\subsection{Exceptional types}
\subsubsection{Hyperspecial}
In Tables 1-5 of the appendix, we demonstrate $$\max\limits_{\substack{y\in P\cap \overline{\mathscr A}_{0}^{Q,n}}}\mathrm{Sat}_{\mathbb G_y}^{\mathbb G}(\mathcal O_{\rm Spr}(j_{W(\Phi_y^{Q,n})}^{W(\Phi_y)}\sgn)(\mathbb C))$$ and the corresponding $ \overline{\Delta}^{Q,n}_y $. We describe $ \overline{\Delta}^{Q,n}_y $ by discarding nodes from the extended Dynkin diagram associated with $ \Delta^{\prime}_0 $ (solid dot means ``discard" and hollow circle means ``reserve"). If there exists $w \in W$ such that $w \Phi^{Q,n}_{y_1}=\Phi^{Q,n}_{y_2}$, we only demonstrate one of them. Comparing with the tables in \cite[\S 4.5]{glt}, we see it is equal to $\mca{O}_{\rm Spr}(j_{W_\nu}^W \sgn)(\mathbb C)$.

\subsubsection{Non-hyperspecial}
In Tables 6-10 of the appendix, we list the weighted Dynkin diagram of each maximal orbit in $ \{ \mca{O}_{\rm Spr}(j_{W(\Phi_y^{Q,n})}^{W(\Phi_v)}\sgn) : y \in v+P\} $ and the corresponding saturation for each non-hyperspecial vertice $v\in \overline{\mathscr A}_0$ up to equivalence. One can check the saturation is not greater than $\mca{O}_{\rm Spr}(j_{W_\nu}^W \sgn)(\mathbb C)$.

\clearpage
\section{Appendix: Exceptional Tables}
% To reduce compilation time of drawing Dynkin diagrams, we directly upload the figures. 
\begin{figure}[h]
	\centering
	\includegraphics[width=0.9\textwidth]{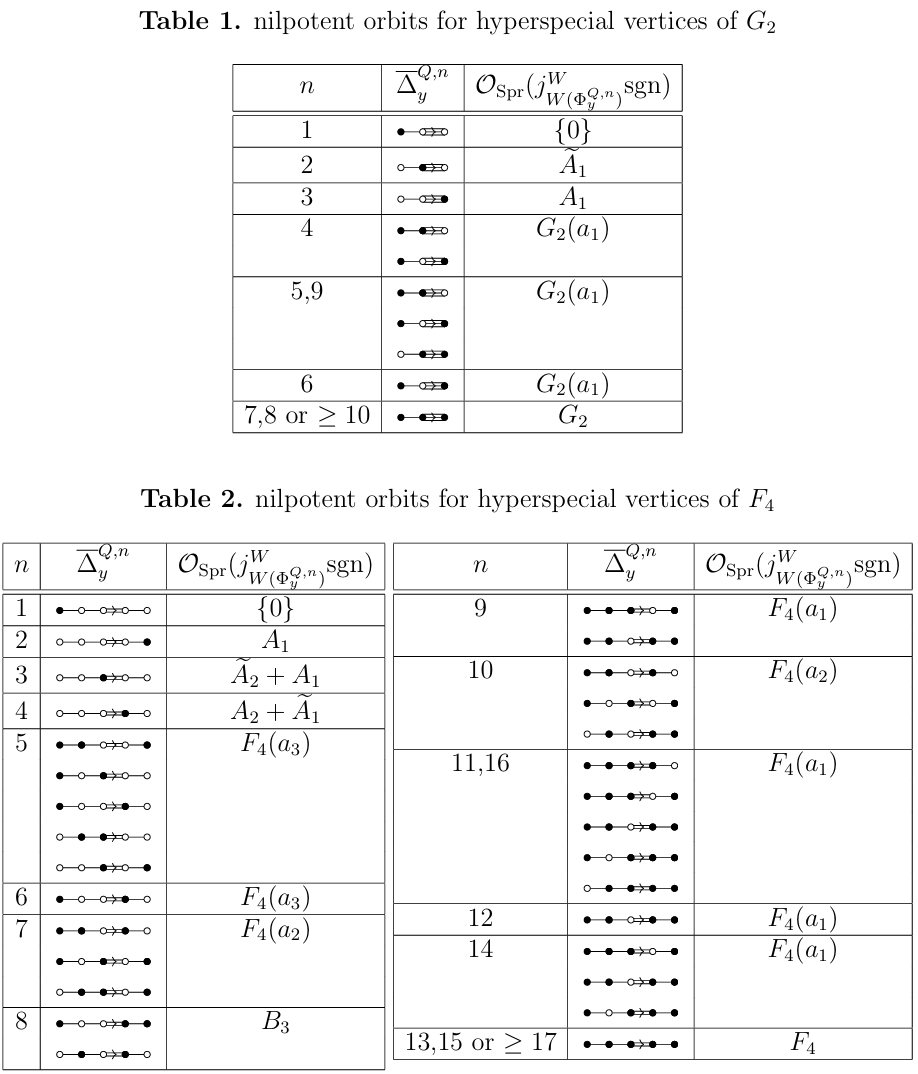}
\end{figure}
\clearpage
\begin{figure}[!htbp]
	\includegraphics[width=0.8\textwidth]{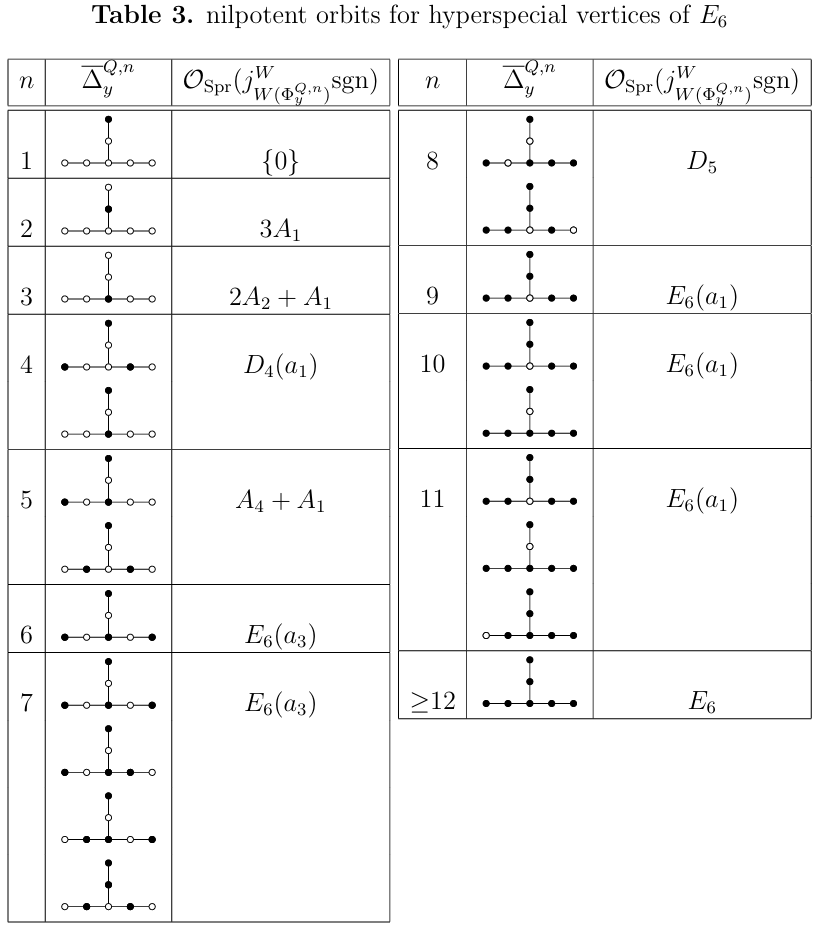}
\end{figure}
\clearpage
\begin{figure}[!htbp]
	\includegraphics[width=0.9\textwidth]{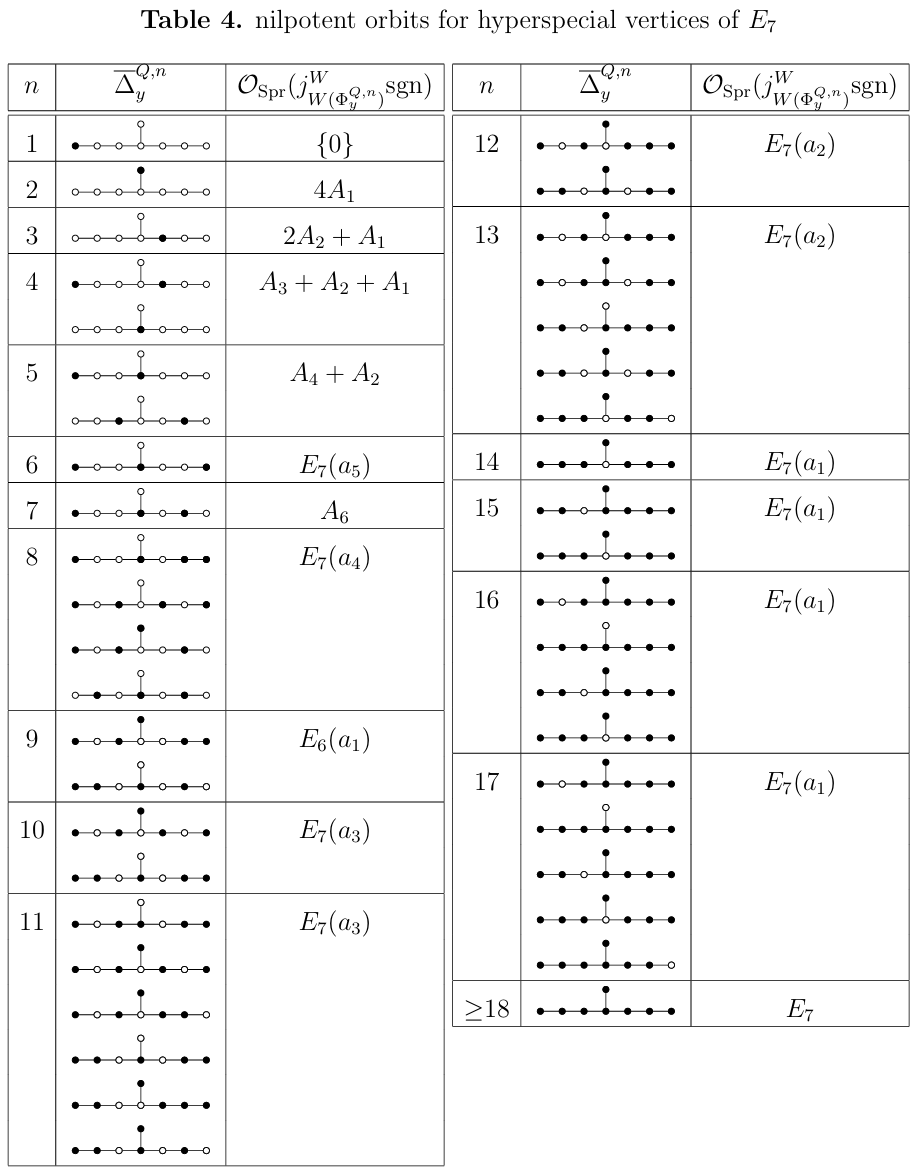}
\end{figure}
\clearpage
\begin{figure}[!htbp]
	\includegraphics[width=0.9\textwidth]{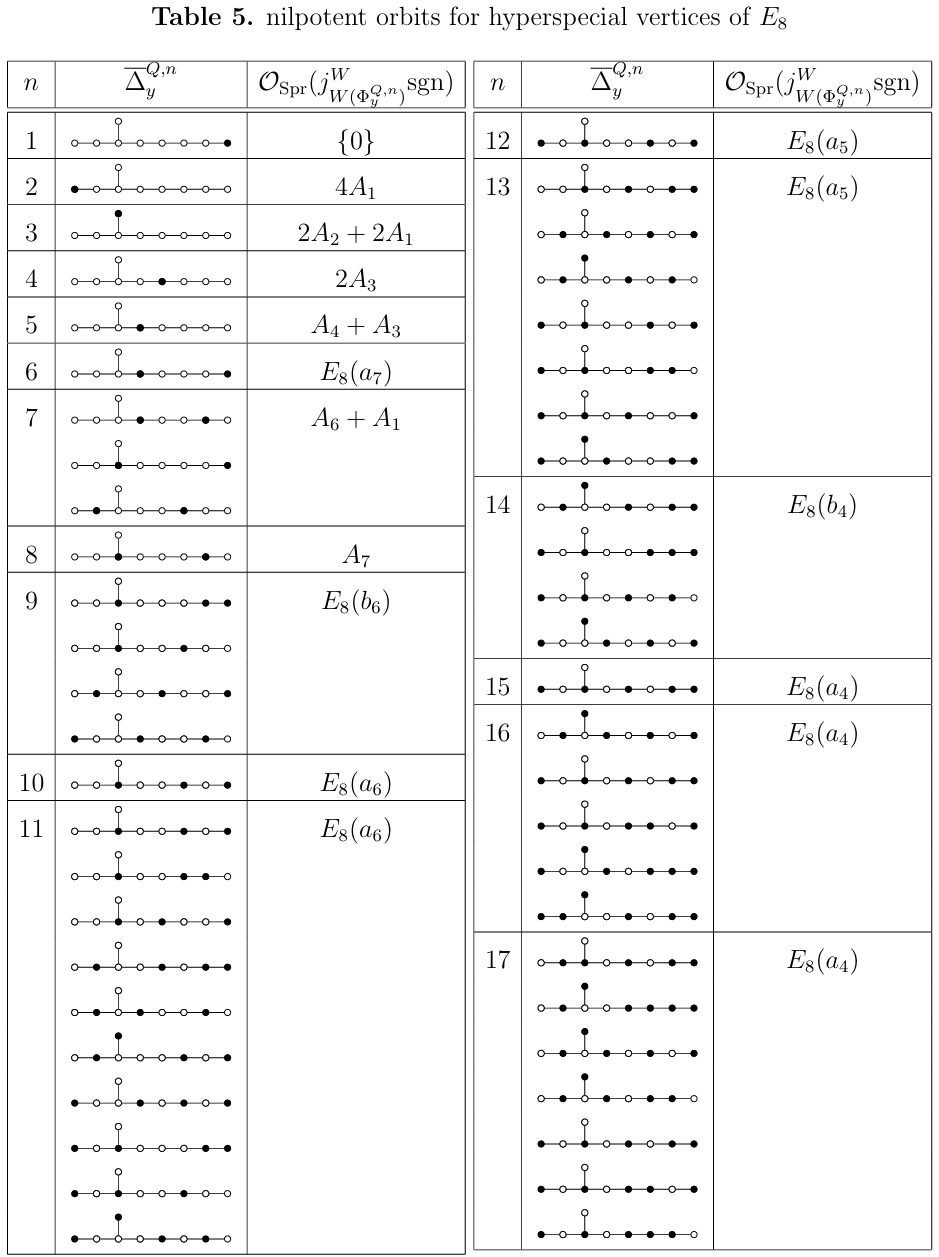}
\end{figure}
\clearpage
\begin{figure}[!htbp]
	\includegraphics[width=0.9\textwidth]{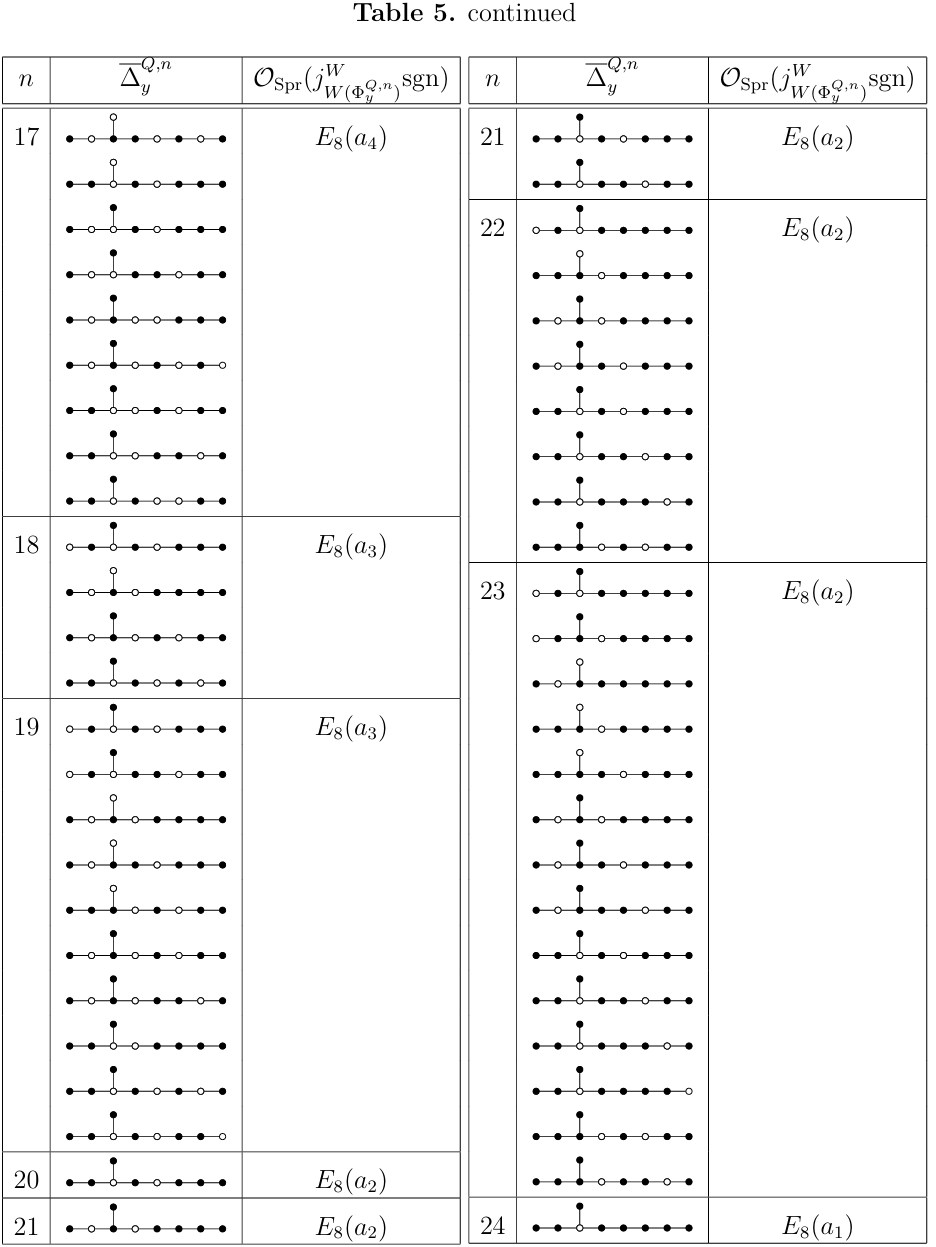}
\end{figure}
\clearpage
\begin{figure}[!htbp]
	\includegraphics[width=0.9\textwidth]{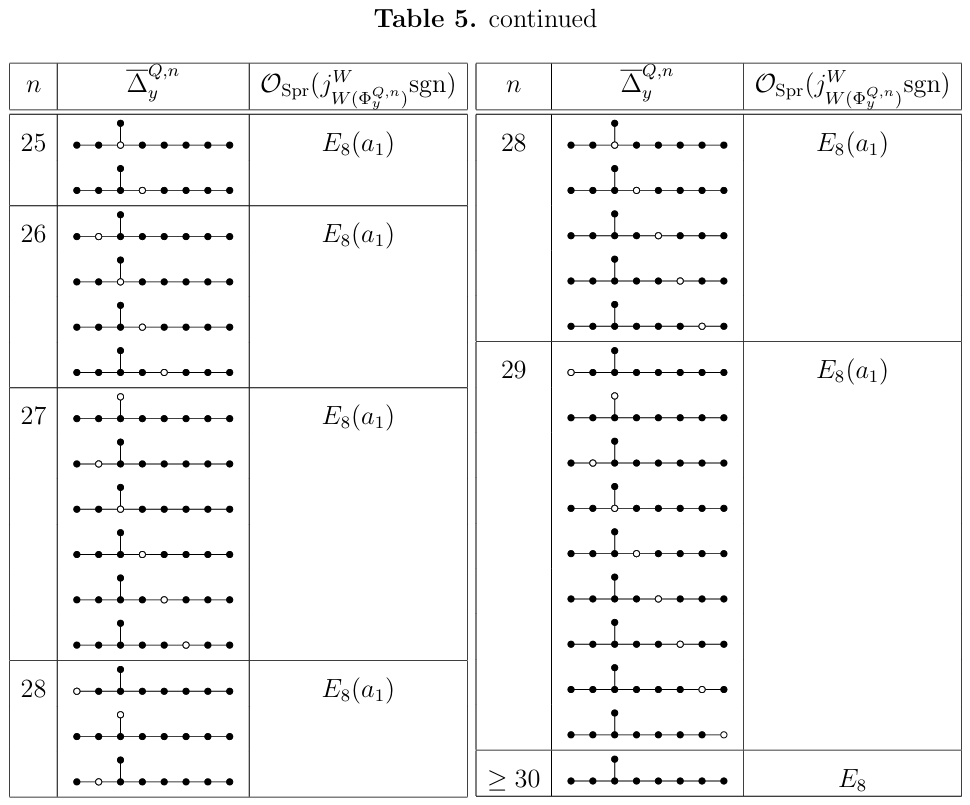}
\end{figure}
\clearpage
\begin{figure}[!htbp]
	\includegraphics[width=0.7\textwidth]{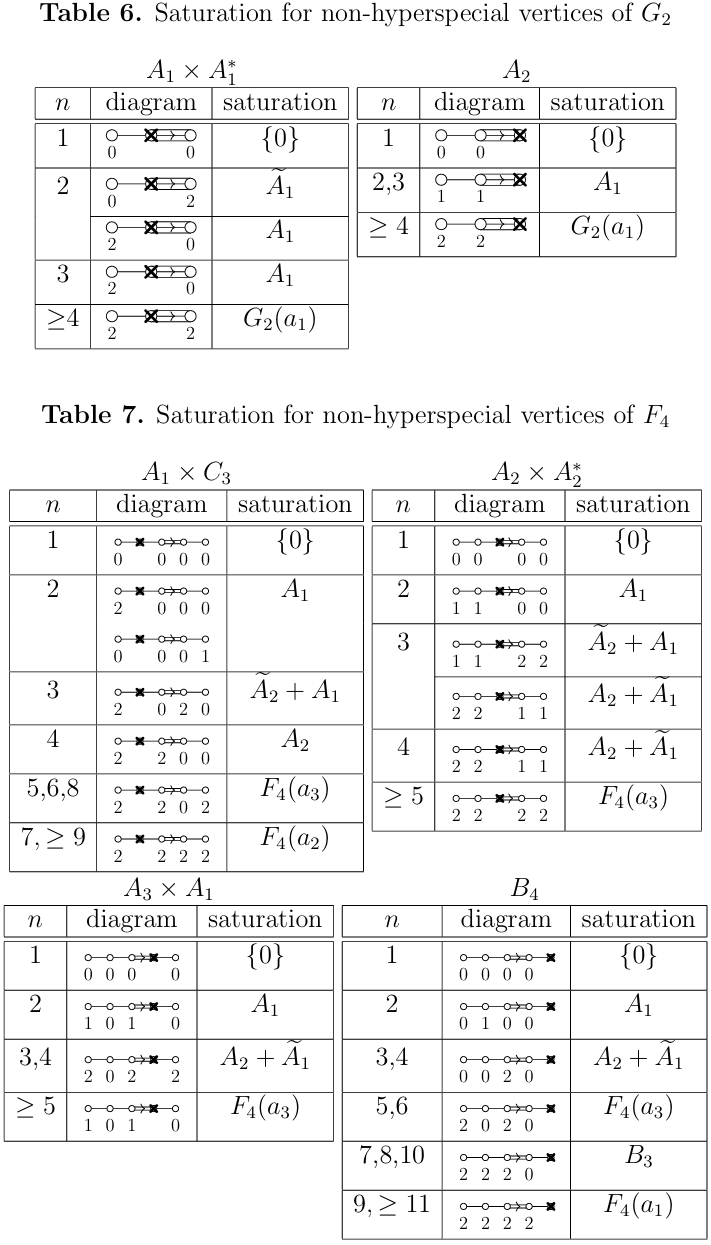}
\end{figure}
\clearpage
\begin{figure}[!htbp]
	\includegraphics[width=0.7\textwidth]{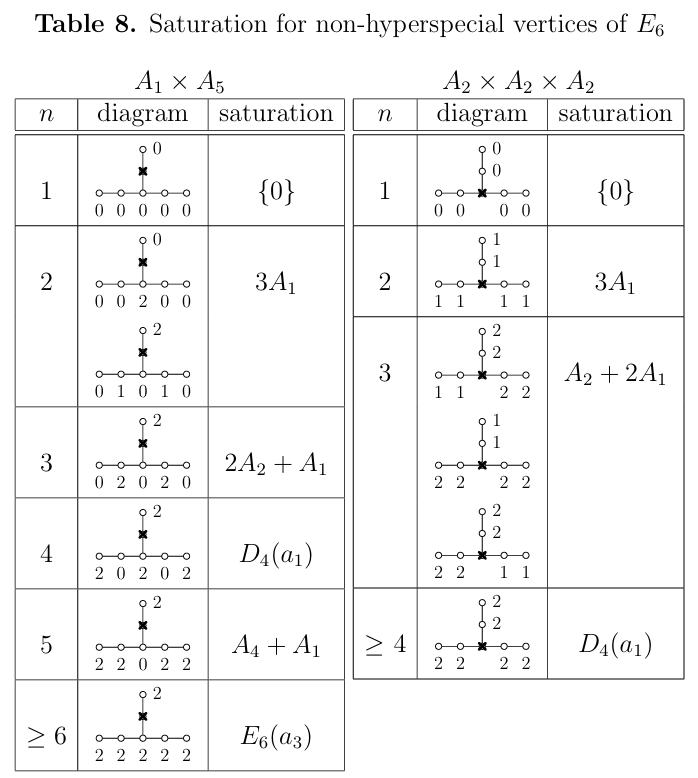}
\end{figure}
\clearpage
\begin{figure}[!htbp]
	\includegraphics[width=0.8\textwidth]{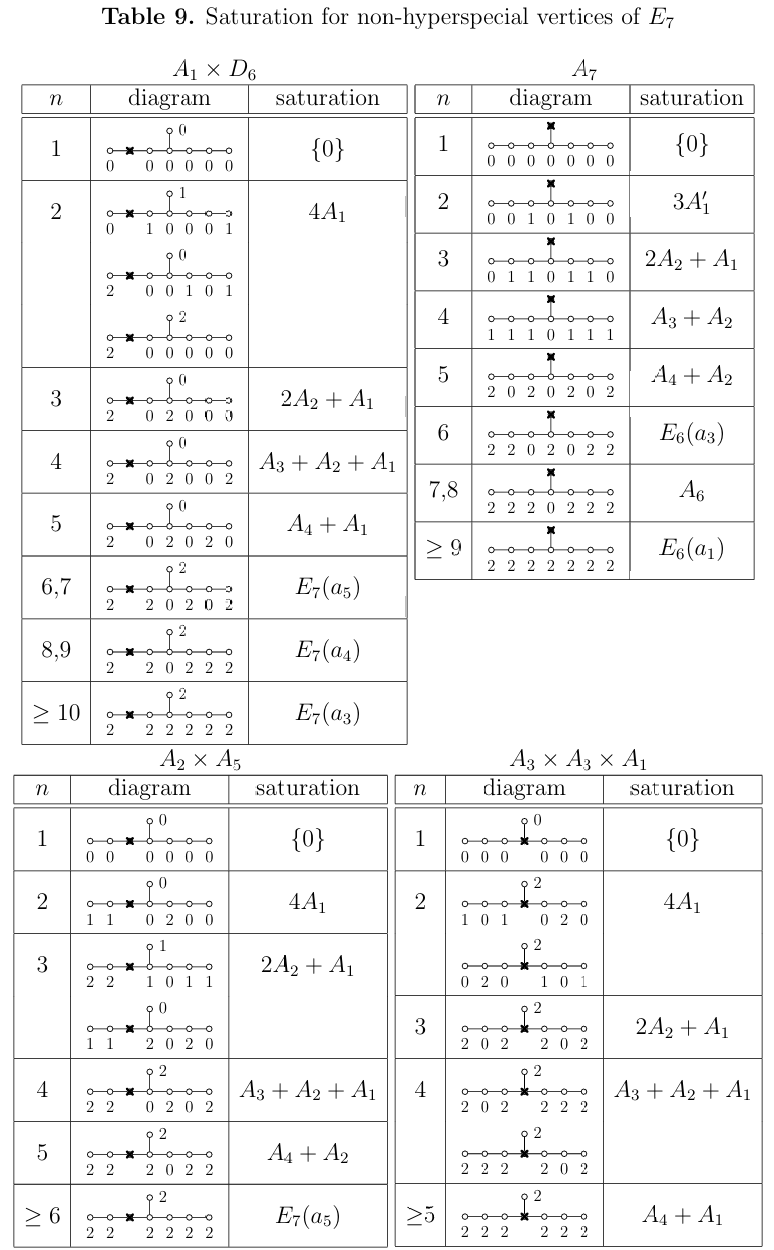}
\end{figure}
\clearpage
\begin{figure}[!htbp]
	\includegraphics[width=0.9\textwidth]{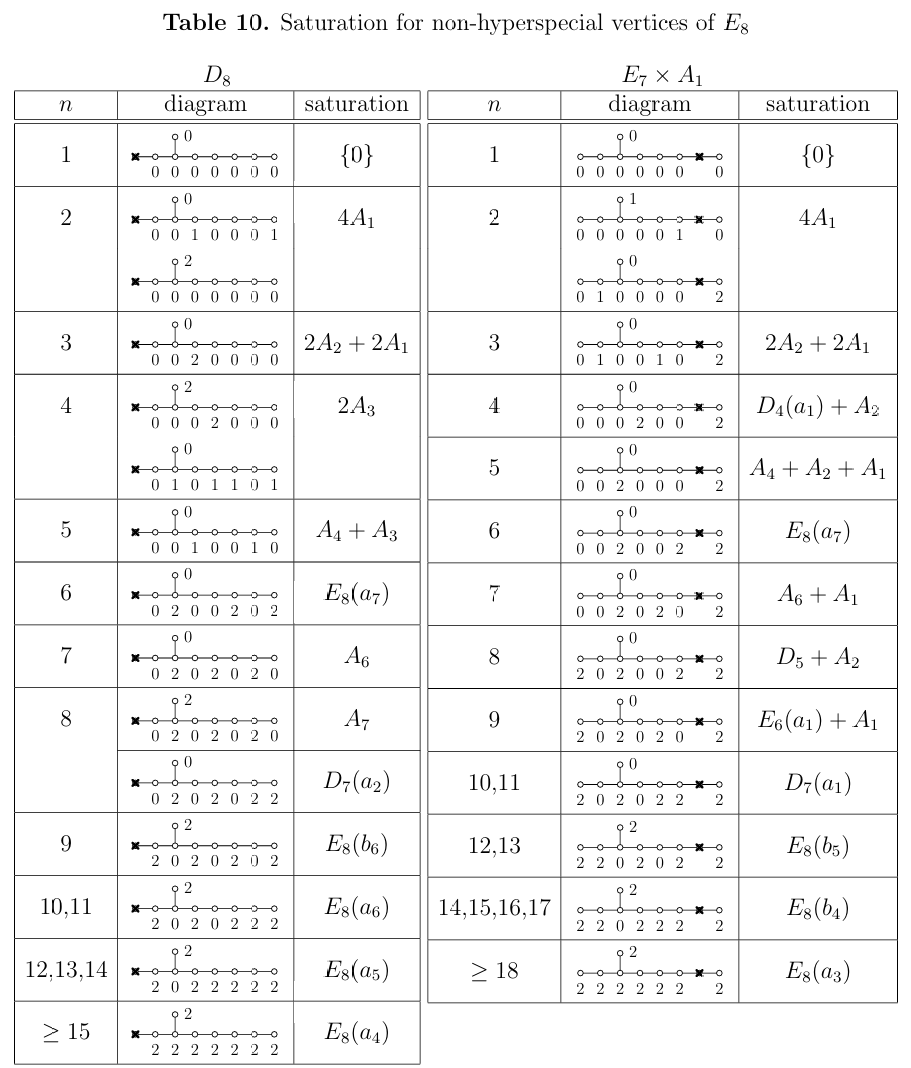}
\end{figure}
\clearpage
\begin{figure}[!htbp]
	\includegraphics[width=0.9\textwidth]{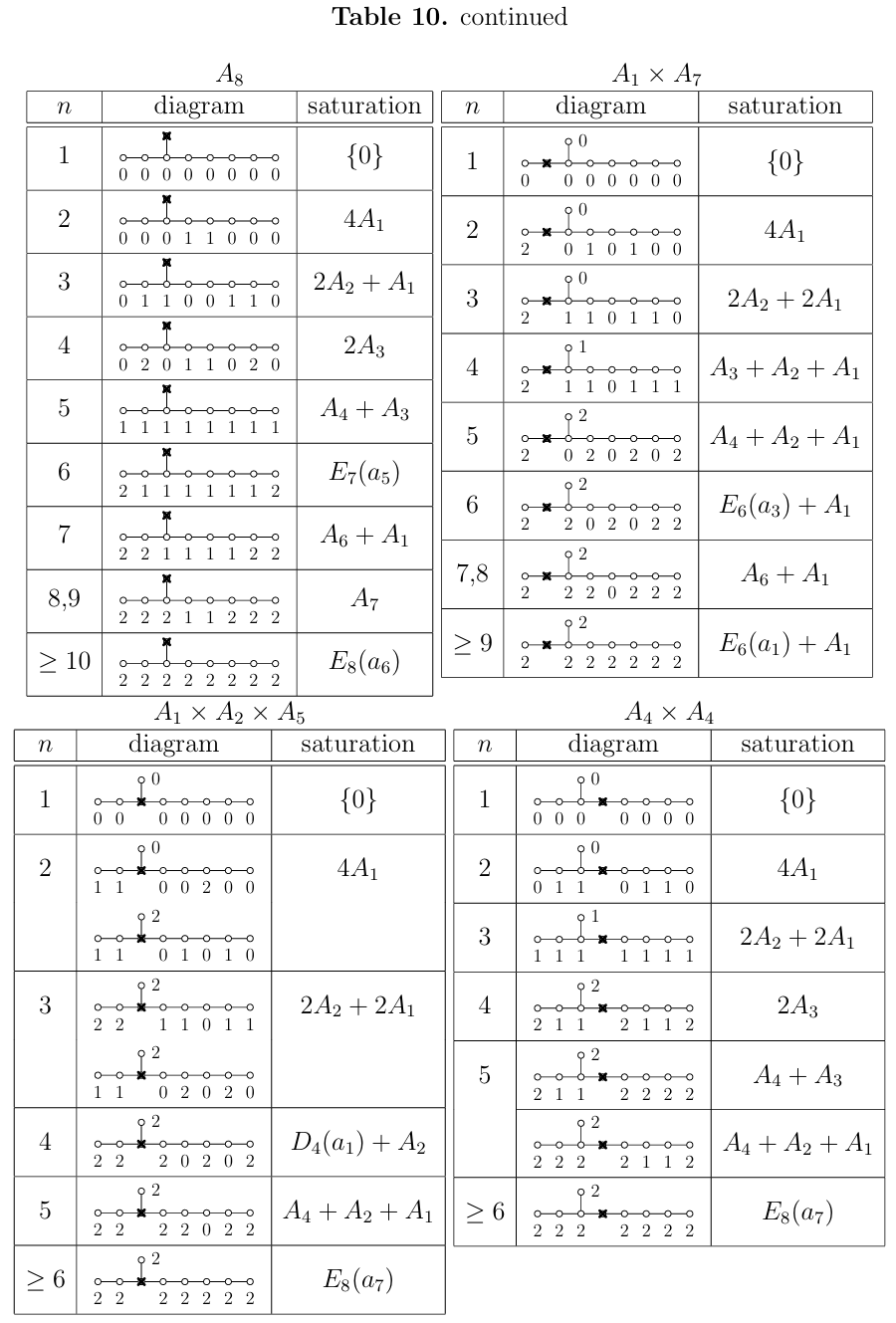}
\end{figure}
\clearpage
\begin{figure}[!htbp]
	\includegraphics[width=0.5\textwidth]{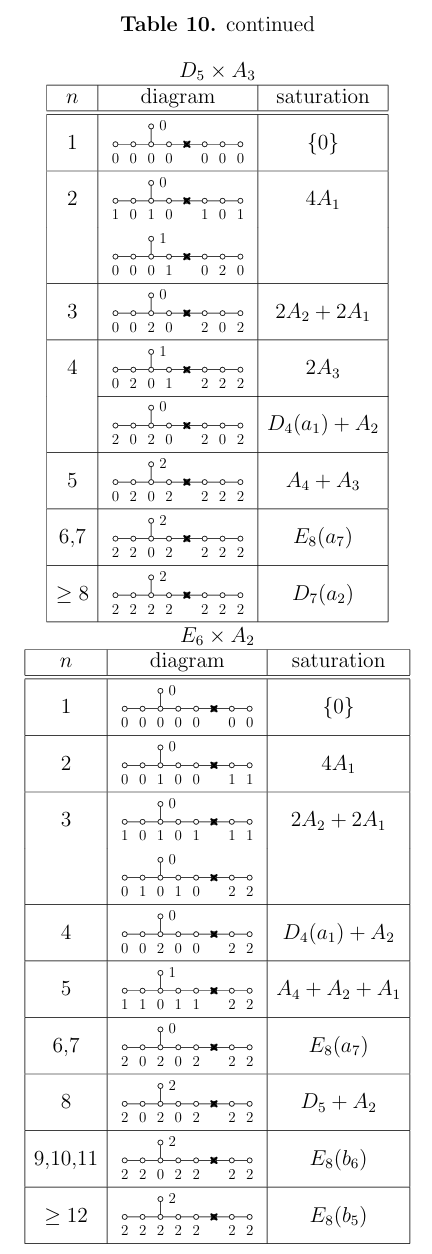}
\end{figure}
\clearpage

\begin{sloppypar} 
\printbibliography[title={References}] 
\end{sloppypar}
\end{document}